\documentclass{article}

\usepackage[a4paper,top=2.54cm,bottom=2.54cm,left=3.17cm,right=3.17cm,%
            includehead,includefoot]{geometry}

\usepackage{amsmath,amssymb,amsfonts,amsthm}
\usepackage{graphicx}
\usepackage{subfigure}
\usepackage{float}
\usepackage{xcolor}
\usepackage[numbers,square,sort&compress]{natbib}
\usepackage{hyperref}
  \hypersetup{colorlinks,citecolor=blue,linkcolor=blue,breaklinks=true}
\usepackage{booktabs}
\usepackage{colortbl}
\usepackage{caption}
\usepackage{enumitem}
\usepackage{epstopdf}
\usepackage{algorithm}
\usepackage{algpseudocode}
\usepackage{array}
\usepackage{bbding}
\usepackage{fancyhdr} 
\usepackage{fancyvrb} 
\usepackage{longtable}
\usepackage{listings}
\usepackage{rotating,rotfloat} 
\usepackage{yhmath} 
\usepackage{amsthm,amsmath,amssymb}
\usepackage{mathrsfs}
\usepackage{amssymb}
\usepackage{indentfirst}
\usepackage{fancyhdr}
\allowdisplaybreaks
\usepackage[all]{xy}
\makeatletter

\newcommand{\Rmnum}[1]{\expandafter\@slowromancap\romannumeral #1@}
\makeatother

\newtheorem{theorem}{Theorem}[section]
\newtheorem{lemma}{Lemma}[section]
\newtheorem{corollary}[theorem]{Corollary}
\newtheorem{proposition}[theorem]{Proposition}
\newtheorem{definition}{Definition}[section]

\newtheorem{pr}[theorem]{Problem}

\newcommand{\bbm}{\begin{bmatrix}}
\newcommand{\ebm}{\end{bmatrix}}

\begin{document}

\title{\textbf{Uniform (d+1)-bundle over the Grassmannian G(d,n)}}

\author{Rong Du \thanks{School of Mathematical Sciences
Shanghai Key Laboratory of PMMP,
East China Normal University,
Rm. 312, Math. Bldg, No. 500, Dongchuan Road,
Shanghai, 200241, P. R. China,
rdu@math.ecnu.edu.cn.
}
and Yuhang Zhou
\thanks{School of Mathematical Sciences
Shanghai Key Laboratory of PMMP,
East China Normal University,
No. 500, Dongchuan Road,
Shanghai, 200241, P. R. China,
727737360@qq.com.
Both authors are sponsored by Innovation Action Plan (Basic research projects) of Science and Technology Commission of Shanghai Municipality (Grant No. 21JC1401900), Natural Science Foundation of Chongqing Municipality, China (general program, Grant No. CSTB2023NSCQ-MSX0334) and Science and Technology Commission of Shanghai Municipality (Grant No. 22DZ2229014).
}
}

\date{}

\maketitle

\begin{center}
	{(Dedicate to the innocent civilians who sacrificed in the war)}
\end{center}

\begin{abstract}
This paper is dedicated to the classification of uniform vector bundles of rank $d+1$ over the Grassmannian $G(d,n)$ ($d\le n-d$) over an algebraically closed field in characteristic $0$. Specifically, we show that all uniform vector bundles with rank $d+1$ over $G(d,n)$ are homogeneous.
\end{abstract}

\section{Introduction}
Algebraic vector bundles over a projective variety $X$ over an algebraically closed field $k$ are basic research objects in algebraic geometry. If $X$ is $\mathbb{P}^1$, then any vector bundle over $X$ splits as a direct sum of line bundles by Grothendieck's well known result. However, if $X$ is a projective space of dimension greater than one, then the structures of vector bundles over $X$ are not so easy to be determined. Since any projective space is covered by lines, it is a natural way to consider the restriction of vector bundles to lines in it. If the splitting type of a vector bundle $E$ keeps same when it restricts to any line in the projective space $\mathbb{P}^n$, then $E$ is called a uniform vector bundle on $\mathbb{P}^n$.
The notion of a uniform vector bundle appears first in Schwarzenberger' paper (\cite{Sch}). Over the field in characteristic zero, much work has been done on the classification of uniform vector bundles over projective spaces. In 1972, Van de Ven (\cite{Ven}) proved that for $n>2$, uniform 2-bundles over $\mathbb{P}^n$ split and uniform 2-bundles over $\mathbb{P}^2$ are precisely the bundles $\mathcal{O}_{\mathbb{P}^2}(a)\bigoplus\mathcal{O}_{\mathbb{P}^2}(b)$,
$T_{\mathbb{P}^2}(a)$ and $\Omega^1_{\mathbb{P}^2}(b)$, where $a,b\in\mathbb{Z}$. In 1976, Sato (\cite{Sat}) proved that for $2<r<n$, uniform $r$-bundles over $\mathbb{P}^n$ split. In 1978, Elencwajg (\cite{Ele}) extended the investigations of Van de Ven to show that uniform vector bundles of rank 3 over $\mathbb{P}^2$, up to dual, are of the forms
$$\mathcal{O}_{\mathbb{P}^2}(a)\bigoplus\mathcal{O}_{\mathbb{P}^2}(b)\bigoplus\mathcal{O}_{\mathbb{P}^2}(c), ~ T_{\mathbb{P}^2}(a)\bigoplus\mathcal{O}_{\mathbb{P}^2}(b)\quad\text{and}\quad S^2T_{\mathbb{P}^2}(a),$$ where $a$, $b$, $c\in \mathbb{Z}$. Previously, Sato (\cite{Sat}) had shown that for $n$ odd, uniform $n$-bundles over $\mathbb{P}^n$ are of the forms
 $$\oplus_{i=1}^{n}\mathcal{O}_{\mathbb{P}^n}(a_i),~ T_{\mathbb{P}^n}(a)\quad\text{and}\quad \Omega^1_{\mathbb{P}^n}(b),$$  where $a_i,a, b\in \mathbb{Z}$. So the results of Elencwajg and Sato yield a complete classification of uniform 3-bundles over $\mathbb{P}^n$. In particular, all uniform 3-bundles over $\mathbb{P}^n$ are homogeneous. Later, Elencwajg, Hirschowitz and Schneider (\cite{E-H-S}) showed that Sato's result is also true for $n$ even. Around 1982, Ellia (\cite{Ell}) proved that for $n+1=4,5,6$, uniform $(n+1)$-bundles over $\mathbb{P}^{n}$ are of the form
  $$\oplus_{i=1}^{n+1}\mathcal{O}_{\mathbb{P}^n}(a_i),~ T_{\mathbb{P}^n}(a)\bigoplus\mathcal{O}_{\mathbb{P}^n}(b)\quad\text{and}\quad\Omega^1_{\mathbb{P}^n}(c)\bigoplus\mathcal{O}_{\mathbb{P}^n}(d), $$ where $a_i, a,b,c,d\in\mathbb{Z}$. Later, Ballico (\cite{Bal}) showed that Ellia's result is still true for any $n$. So, uniform $(n+1)$-bundles over $\mathbb{P}^{n}$ are totally classified. In particular, they are all homogeneous. One can see from Ellia and Ballico's papers that classification problem becomes harder when the rank of a vector bundle is great that then $n$.

Uniform bundles over an algebraically closed field in characteristic $0$ are widely studied not only on projective spaces, but also on special Fano manifolds of Picard number one (\cite{Bal2} \cite{K-S} \cite{Guy} \cite{D-F-G} \cite{D-F-G2} \cite{M-O-C2} \cite{Pan}).  In \cite{M-O-C3}, Mu˜\~{n}oz, Occhetta and Sol\'{a} Conde proposed a problem as follows.

\begin{pr}
Classify low rank uniform principal $G$-bundles ($G$ semisimple algebraic group) on rational homogeneous spaces.
\end{pr}

Grassmannians are simplest rational homogeneous spaces except projective spaces. In 1985, Guyot (\cite{Guy}) proved that when $r<d$, uniform $r$-bundles over $G(d,n)$ ($d\leq n-d$) split and when $r=d$, are one of the form:

\begin{center}
	$H_d{(a)}, H_d^*{(b)} ~\text{and}~ \overset{k}{\underset{i=1}{\oplus}} \mathcal{O}_{G(d,n)}{(c_i)},$
\end{center}
where $a,b,c_i \in \mathbb{Z}$. For $r>d$, the classification problem still remains open.


In this paper, we classify uniform $(d+1)$-bundles over $G(d,n)$ ($d\le n-d$) over an algebraic closed field in characteristic $0$ and deduce the following main theorem.
\begin{theorem}\label{char 0}
	Let $E$ be a uniform $(d+1)$-bundle over the Grassmannian $G(d,n)$ over an algebraically closed field in characteristic zero, where $2 \le d\le n-d$. Then $E$ is isomorphic to one of the following:
	\begin{center}
	$H_d{(a)}\bigoplus \mathcal{O}_{G(d,n)}{(b)}, H_d^*{(c)}\bigoplus \mathcal{O}_{G(d,n)}{(d)},\overset{k}{\underset{i=1}{\oplus}} O^{\oplus r_{i}}_{G(d,n)}{(e_i)}, Q_{n-d}(s), Q_{n-d}^*(t)~\text{and}~S^2 H_2(f)$
	\end{center}	
where $a,b,c, d, e_i, s, t, f \in \mathbb{Z}$ and $H_d$ is the universal subbundle over  $G(d,n)$.
\end{theorem}

\begin{corollary}
All uniform vector bundles with rank $d+1$ over $G(d,n)$ over an algebraically closed field in characteristic $0$ are homogeneous.
\end{corollary}


\section{Preliminaries}
\subsection{Grassmannian and Flag variety}
Let $V=k^n$, where $k$ is an algebraically closed field. Denote $G(d,n)$ ($d\le n-d$) to be the Grassmannian manifold of $d$-dimensional linear subspace of $V$. Let $\mathcal{V}:=G(d,n)\times V$ be the trivial vector bundle of rank $n$ on $G(d,n)$ whose fiber
at every point is the vector space $V$. We write $H_d$ for the rank $d$ subbundle of $\mathcal{V}$ whose
fiber at a point $[\Lambda]\in G(d,n)$ is the subspace $\Lambda$ itself; that is,
\[(H_d)_{[\Lambda]}=\Lambda\subseteq V=\mathcal{V}_{[\Lambda]}.\]
$H_d$ is called the \emph{universal subbundle} on $G(d,n)$; the rank $n-d$ quotient bundle $Q_{n-d}=\mathcal{V}/H_d$ is called the \emph{universal quotient bundle}, i.e., they fit the exact sequence
\begin{equation}\label{tau}
0\rightarrow H_d\rightarrow \mathcal{V}\rightarrow Q_{n-d}\rightarrow 0.
\end{equation}

%

The flag variety $F(d_1,d_2,\cdots,d_s,n)$ is a manifold parameterizing increase chain of k-linear subspace in $V$:

\begin{center}
	$F(d_1,d_2,\cdots,d_s,n)=\{(A_1,\cdots A_{s})\in G(d_{1},n)\times \cdots G(d_{s},n)|A_1\subset A_2 \subset \cdots \subset A_s, \text{dim}A_i=d_i, i=1, \dots, s\} $.
\end{center}

When $s=n,d_{i}=i$, we call flag variety $F(1,2,\cdots, n-1,n)$ the complete flag variety. There are universal subbundles $H_{d_i}(1\leq i\leq s)$ of rank $d_i$ over $F(d_1,d_2,\cdots,d_s,n)$ whose fibers at the point $(A_1,\cdots A_{s})$ is $A_{d_{i}}$. In fact, there are natural projections from $F(d_1,d_2,\cdots,d_s,n)$ to $G(d_{i},n)$. The pull back of the universal subbundle $H_{d_i}$ by the projection map is the universal subbundle over $F(d_1,d_2,\cdots,d_s,n)$, which is still denoted by $H_{d_i}$ if there is no confusion in the context. In particular, we have universal subbundles $H_{i}, 1\leq i\leq n$ over $F(1,2,\cdots, n-1,n)$ and denote $X_{i}:=c_{1}(H_{i}/H_{i-1})$.

There are also universal quotient bundles $Q_{n-d_i}(1\leq i\leq s)$ of rank $n-d_i$ over $F(d_1,d_2,\cdots,d_s,n)$ whose fibers at the point $(A_1,\cdots A_{s})$ is $V/A_{d_{i}}$.

The Chow ring of a flag variety can be found in \cite{Guy} Theorem 3.2.

\begin{theorem}[\cite{Guy} Theorem 3.2]\label{Chow}
	The Chow ring $A(F)$of flag variety $F=F(d_1,d_2,\cdots,d_k,n) $ is
	
\begin{center}
		$Z[X_1,\cdots,X_{d_1};X_{d_{1}+1},\cdots X_{d_2}; \cdots; X_{d_{k}+1},\cdots,X_{n}]/I$,
\end{center}
	where $Z[X_1,\cdots,X_{d_1};X_{d_{1}+1},\cdots X_{d_2}; \cdots; X_{d_{k}+1},\cdots,X_{n}]$ is a ring of polynomials that are symmetric about $X_{d_i+1},\cdots,X_{d_{i+1}}$, where $0\leq i\leq k-1$, if we assume $d_0=0$, and $I$ is an ideal generated by $$\underset{a_1+a_2+\cdots+a_{d_k}=j}{\sum}X_1^{a_1}X_2^{a_2}\cdots X_{d_k}^{a_{d_k}},$$ where $a_{t}\ge 0$ and $(n-d_k) < j \leq n$.
\end{theorem}

\subsection{Standard diagram}

We have the standard commutative diagram as follows:
\begin{align}
	\xymatrix
	{
		F(d-1,d,n)\ar[dr]^{q} & F(d-1,d,d+1,n)\ar[r]^{pr_2} \ar[d]^{p}\ar[l]_{pr_1}& F(d,d+1,n)\ar[dl]_{r}\\
		& G(d,n), & \\
	}
\end{align}
where all the morphisms in the diagram are projections.

The mapping $q~(resp.~r)$ identifies $F(d-1,d,n)~(resp.~F(d,d+1,n)$ with the projective bundle $\mathbb{P}(H_{d})~(resp.~\mathbb{P}(Q_{n-d}^{*}))$ of $G(d,n)$.
Let $\mathscr{H}_{H_{d}^{*}}~(resp.~  \mathscr{H}_{Q_{n-d}})$ be the tautological line bundle on $F(d-1,d,d+1,n)$ associated to $F(d-1,d,n)~(resp.~ F(d,d+1,n))$, i.e.
\[\mathscr{H}_{H_{d}^{*}}={pr_1}^{*}\mathcal{O}_{F(d-1,d,n)}(-1)~(resp.~ \mathscr{H}_{Q_{n-d}}={pr_2}^{*}\mathcal{O}_{F(d,d+1,n)}(-1)).\]
So
\[{pr_1}_{*}\mathscr{H}_{H_{d}^{*}}=\mathcal{O}_{F(d-1,d,n)}(-1)~(resp.~ {pr_2}_{*}\mathscr{H}_{Q_{n-d}}=\mathcal{O}_{F(d,d+1,n)}(-1)).\]

%
%

Moreover, we also have the standard diagram
\begin{align}
	\xymatrix{
		F(d-1,d,d+1,n)\ar[r]^{s}\ar[d]^{p}&F(d-1,d+1,n)\\
		G(d,n).&\\}
\end{align}

The map $s$ identifies $F(d-1,d,d+1,n)$ with the projective bundle $\mathbb{P}(H_{d+1}/H_{d-1})$ over $F(d-1,d+1,n)$. Denote $\mathcal{O}_{s}(1)$ to be the relative Hopf bundle associate to $H_{d+1}/H_{d-1}$. In fact, $\mathcal{O}_{s}(1)=\mathscr{H}_{H_d^*}^*$.

\subsection{relative HN-filtration}

Let $E$ be a vector bundle of rank $r$ on $G(d,n)$ and $L$ be a line over $G(d,n)$. Suppose that $E|_{L}=\overset{k}{\underset{i=1}{\oplus}} O^{\oplus r_i}_{L}{(u_i)}$, where $u_1>u_2>\cdots >u_k$. If $k,r_1,r_2,\cdots,r_k,u_1,u_2,\cdots,u_k$ are independent of $L$, then $E$ is called uniform of splitting type  $(k,r_1,r_2,\cdots,r_k;u_1,u_2,\cdots,u_k)$.

For any uniform vector bundles E of type $(k,r_1,r_2,\cdots,r_k;u_1,u_2,\cdots,u_k)$ over $G(d,n)$,there exist a relative HN-filtration \cite{E-H-S} $HN^{i}$ of $p^{*}E$ over $F(d-1,d,d+1,n)$:

\begin{center}
	$0\subset HN^{1}(E)\subset HN^{2}(E)\subset \cdots \subset HN^{k}(E)=p^{*}E,$
\end{center}

where $HN^{i}(E)=Im[s^{*}s_{*}(p^{*}E\otimes \mathcal{O}_{s}(-u_i))\otimes \mathcal{O}_{s}(u_i)\longrightarrow p^{*}E]$ and $\mathcal{O}_{s}(u)=O^{\otimes u}_{s}(1)$.

Moreover, there are exact sequences:

\begin{center}
	$0\rightarrow HN^{i}(E)\rightarrow HN^{i+1}(E)\rightarrow s^{*}(E_i)\otimes \mathcal{O}_s(u_i)\rightarrow 0$,
\end{center}
where $E_i$ is a vector bundle of rank $r_i$ over $F(d-1,d+1,n)$.

\begin{definition}
	The Chern polynomial of a vector bundle $E$ of rank r over variety X is
\begin{center}
		$c_{E}(T)=T^r-c_1(E)T^{r-1}+c_2(E)T^{r-2}+\cdots+(-1)^{r}c_{r}(E)$
\end{center}
	
	where $c_{i}(E)$ is the $i$-th Chern class of $E$.
\end{definition}

\begin{lemma}[\cite{Guy}]\label{Guy}
	The Picard group of $F(d-1,d,d+1,n)$ is generated by $ p^*{\mathcal{O}_G(1)}, \mathscr{H}_{H_d^*}$ and $\mathscr{H}_{Q_{n-d}}$, and their Chern polynomials are $ T+(X_{1}+X_{2}+\cdots+X_{d}) ,T+X_d$ and $T-X_{d+1} $ respectively.
\end{lemma}
The following lemma which can be found in \cite{E-H-S} will be used in our proofs.
\begin{lemma}[\cite{E-H-S} Proposition 3.5]\label{1-1lemma}
	Let $X$ be a projective manifold and $K$ be a vector bundle over $X$.  Suppose that $p$ is a morphism from $F=\mathbb{P}(K)$ to $X$ and $\mathscr{H}_K$ is the relative Hopf bundle over $F$. Then a vector bundle $E$ over $X$ has a subbundle isomorphic to $K$ if and only if $p^*E $ has a subbundle isomophic to $\mathscr{H}_K$.
\end{lemma}

\section{(d+1)-uniform bundle over $G(d,n)$ when $d<n-d-1$}
From now on, we suppose that everything is over an algebraically closed fields in characteristic $0$.

First, we consider the Chern polynomial of $p^{*}E$ and its relative HN-filtration.
By Lemma \ref{Guy},
$c_{p^{*}E}(T)$ and $c_{s^{*}(E_i)}(T)$ are both homogeneous polynomials. Moreover, $c_{p^{*}E}(T)$ is symmetric about $X_1,X_2,\cdots X_d$ and $c_{s^{*}(E_i)}(T)$ is symmetric about $X_1,X_2,\cdots X_{d-1} $ and $ X_d,X_{d+1}$ respectively by Theorem \ref{Chow}.

 With the help of the relative HN-filtration, we deduce that  $c_{p^{*}E}(T)=\overset{k}{\underset{i=1}{\prod}}c_{s^{*}(E_i)\otimes \mathcal{O}_{s}(u_j)}(T)$  in the Chow ring $A(F(d-1,d,d+1,n))$, i.e,
 \begin{equation}\label{E=S}
 	c_{p^{*}E}(T)=\overset{k}{\underset{i=1}{\prod}}c_{s^{*}(E_i)\otimes \mathcal{O}_{s}(u_j)}(T)~~~(\text{mod} ~I).
 \end{equation}


\begin{lemma}[\cite{Guy} Proposition 4.1]\label{ES}
	If $$E(T;X_1,X_2,\cdots X_d)=\overset{k}{\underset{i=1}{\prod}}S_i(T+u_i X_d;X_1,X_2,\cdots X_{d-1};X_d,X_{d+1}),$$
	where $E(T;X_1,X_2,\cdots X_d)$ is a homogeneous polynomial symmetric about $X_1,X_2,\cdots X_d$ and
	$S_i(T;X_1,X_2,\cdots X_{d-1};X_d,X_{d+1})$ is a homogeneous polynomial symmetric about $X_1,X_2,\cdots X_{d-1} $ and $ X_d,X_{d+1}$ respectively.	
	Then every irreducible factor of $E(T;X_1,X_2,\cdots X_d)$ has the form  $T+\sum_{i=1}^{d}\lambda_i X_i$, where each $\lambda_i$ is a constant.
\end{lemma}

\subsection{Chern Polynomial of $p^{*}E$}
When the uniform vector bundle $E$ is of rank $d+1$ under the assumption $d<n-d-1$, the Chern polynomial $c_{p^{*}E}(T)$ is of degree $d+1$ and degree of non-zero elements in $I$ are of degrees great than $n-d>d+1$ by Theorem \ref{Guy}. So the equation (\ref{E=S}) is exactly
\[
 c_{p^{*}E}(T)=\overset{k}{\underset{i=1}{\prod}}c_{s^{*}(E_i)\otimes \mathcal{O}_{s}(u_j)}(T).\]
By Lemma \ref{ES}, $c_{p^{*}E}(T)$ has an irreducible polynomial of form:

\begin{center} $T+\lambda_{1}(\underbrace{X_{i_1}+\cdots+X_{j_1}}_{k_1})+\lambda_{2}(\underbrace{X_{i_2}+\cdots+X_{j_2}}_{k_2})+\cdots+\lambda_{l}(\underbrace{X_{i_l}+\cdots+X_{j_l}}_{k_l}),$
\end{center}

where $k_1+k_2+\cdots +k_{l}=d$ and $\lambda_{1}, \dots, \lambda_{l}$ are all distinct. Since $c_{p^{*}E}(T)$ is symmetric about $X_1,X_2,\cdots X_d$, its degree is at least $ \dfrac{d!}{{k_1}!{k_2}!\cdots{k_l}!} $.

\begin{proposition}\label{polynomial when d<n-d-1}
Suppose that $$E(T;X_1,X_2,\cdots X_d)=\overset{k}{\underset{i=1}{\prod}}S_i(T+u_i X_d;X_1,X_2,\cdots X_{d-1};X_d,X_{d+1}),$$
where $E(T;X_1,X_2,\cdots X_d)$ is a homogeneous polynomial of degree $d+1$ and symmetric about $X_1,X_2,\cdots X_d$
, and
$S_i(T;X_1,X_2,\cdots X_{d-1};X_d,X_{d+1})$ is a homogeneous polynomial symmetric about $X_1,X_2,\cdots X_{d-1} $ and $ X_d,X_{d+1}$ respectively.
If the degree of $E(T;X_1,X_2,\cdots X_d)$ is $d+1$, then
$$E(T;X_1,X_2,\cdots X_d)=\overset{k}{\underset{i=1}{\prod}}(T+\lambda_{i}(X_{1}+X_{2}+\cdots X_{d}))^{r_i} $$
	or
\begin{multline*}
	E(T;X_1,X_2,\cdots X_d)
	=\overset{d}{\underset{i=1}{\prod}}(T+\lambda_{1}X_{i}+\lambda_{2}(X_{1}+X_{2}+\cdots\\+\overset{\wedge}{X_i}+\cdots+ X_{d}))(T+\lambda'_{1}(X_{1}+X_{2}+\cdots+X_{d})).
\end{multline*}
\end{proposition}
\begin{proof}
By Lemma \ref{ES}, suppose that an irreducible polynomial of $E(T;X_1,X_2,\cdots X_d)$ is of form
\begin{center} $T+\lambda_{1}(\underbrace{X_{i_1}+\cdots+X_{j_1}}_{k_1})+\lambda_{2}(\underbrace{X_{i_2}+\cdots+X_{j_2}}_{k_2})+\cdots+\lambda_{l}(\underbrace{X_{i_l}+\cdots+X_{j_l}}_{k_l}),$
\end{center}
where $k_1+k_2+\cdots +k_{l}=d$ and $\lambda_{1}, \dots, \lambda_{l}$ are all distinct.
\begin{eqnarray*}
&&\dfrac{d!}{{k_1}!{k_2}!\cdots{k_l}!}\\
&=&\dfrac{(k_1+k_2-1)!}{k_1!k_2!}\cdot
\dfrac{(k_1+k_2+k_3-2)!}{(k_1+k_2-1)!k_3!}\cdot\cdots\cdot \dfrac{(k_1+\cdots+k_l-(l-1))!}{(k_1+\cdots+k_{l-1}-(l-2))!k_l!}\cdot\dfrac{d!}{(d-(l-1))!}\\
&\ge&\dfrac{d!}{(d-(l-1))!}.
\end{eqnarray*}

%
%
%
If $l\geq 3 $, then $ d(d-1)\leq \dfrac{d!}{(d-(l-1))!}\leq\dfrac{d!}{{k_1}!{k_2}!\cdots{k_l}!} \leq d+1 $ which is a contradiction since $ d\geq l\geq3$. Thus, $ l\leq2$.

Let's consider the case $l=2$ first. Without loss of generality, we may assume $k_{1}\leq k_{2}$.
%
If $k_1\geq 2$, then $\dfrac{3}{2}d\leq \dfrac{d!}{k_{1}!k_{2}!}\leq (d+1)$. So $d\leq 2 $ contradicts to $ d\geq 1+k_{1}\geq 3 $. So $ k_{1}=1$ and
$\overset{d} {\underset{i=1}{\prod}}(T+\lambda_{1}X_{i}+\lambda_2(X_{1}+X_{2}+\cdots+\overset{\wedge}{X_i}+\cdots+ X_{d})) $ divides $E(T;X_1,X_2,\cdots X_d)$ since $E(T;X_1,X_2,\cdots X_d)$ is symmetric about $X_1,X_2,\cdots X_d$. Comparing the degree of these two polynomials, we get the second identity in the statement.

If $l=1$, then $ E(T;X_1,X_2,\cdots X_d) $ has a factor $T+\lambda'_{1}(X_{1}+X_{2}+\cdots+X_{d} )$.
Assume that
\begin{center}
 	$E(T;X_1,X_2,\cdots X_d)=(T+\lambda'_{1}(X_{1}+X_{2}+\cdots+X_{d}))E'(T;X_1,X_2,\cdots X_d).$
 \end{center}
Obviously, $E'(T;X_1,X_2,\cdots X_d)$ satisfies the condition of Lemma \ref{ES}. So $E(T;X_1,X_2,\cdots X_d)$ is of the form

$$\overset{d}{\underset{i=1}{\prod}}(T+\lambda_{1}X_{i}+\lambda_2(X_{1}+X_{2}+\cdots+\overset{\wedge}{X_i}+\cdots+ X_{d}))(T+\lambda'_{1}(X_{1}+X_{2}+\cdots+X_{d}))$$ or

$$\overset{k}{\underset{i=1}{\prod}}(T+\lambda_{i}(X_{1}+X_{2}+\cdots X_{d}))^{r_i}$$ by induction.
\end{proof}

In \cite{Guy} (Proposition 2.5), Guyot proved that the uniform vector bundle $E$ with split type $(k;u_1,\cdots,u_k;r_1,\cdots,r_k)$ is an extension of two uniform vector bundles if $ u_i-u_{i+1}\geq2 $ for some $i$. So we only need to classify the uniform bundle when $u_i-u_{i+1}=1$ for all $i$. Now we are going to classify the uniform vector bundle with different Chern polynomials and prove Theorem \ref{char 0}.

\begin{proposition}\label{juti ploynimal when d<n-d}
Let $E$ be a uniform vector bundle of rank $d+1$ over $G(d,n)$.  If $u_i-u_{i+1}=1$ for all $i$, then the Chern polynomial of $p^*E$ (taking a dualization or tensoring with suitable line bundles) are one of the following cases.
	
	1. $c_{p^{*}E}(T)=\overset{k}{\underset{i=1}{\prod}}(T+u_{i}(X_{1}+X_{2}+\cdots X_{d}))^{r_i}$;
	
	2. $c_{p^{*}E}(T) =\overset{d}{\underset{i=1}{\prod}}(T+X_{i})(T-(X_{1}+X_{2}+\cdots+X_{d}) $;
	
	
	3. $ c_{p^*E}(T)=\overset{d}{\underset{i=1}{\prod}}(T+X_{i})T $;
	
	
	4. $ c_{p^*E}(T)=\overset{d}{\underset{i=1}{\prod}}(T+X_{i})(T+(X_{1}+X_{2}+\cdots+X_{d})) $;
	
	
	5. $ c_{p^*E}(T)=\overset{d}{\underset{i=1}{\prod}}(T+X_{i})(T+2(X_{1}+X_{2}+\cdots+X_{d})) $;
	
	
	6.	$ c_{p^*E}(T) =\overset{d}{\underset{i=1}{\prod}}(T+2X_{i})(T+(X_{1}+X_{2}+\cdots+X_{d})) $.
	
\end{proposition}

\begin{proof}
With the result of Proposition \ref{polynomial when d<n-d-1}, $ c_{p^*E}(T)$ is of the form $$\overset{k}{\underset{i=1}{\prod}}(T+u_{i}(X_{1}+X_{2}+\cdots X_{d}))^{r_i}$$ or
\begin{center}
	$\overset{d}{\underset{i=1}{\prod}}(T+\lambda_{1}X_{i}+\lambda_2(X_{1}+X_{2}+\cdots+\overset{\wedge}{X_i}+\cdots+ X_{d}))(T+\lambda_{3}(X_{1}+X_{2}+\cdots+X_{d})$,
\end{center}
where $\lambda_{1},\lambda_{2},\lambda_{3}$ are one of $u_1,u_2,u_3$. After taking dualization of $E$ or tensoring $E$ with $\mathcal{O}_{G(d,n)}(-\lambda_{2})$, the Chern polynomial of $p^*E$ must be one in the statement.
\end{proof}

\subsection{classification of $(d+1)$-uniform bundle over $G(d,n)$ when $d<n-d-1$}
We will consider the Chern polynomials in the Proposition \ref{juti ploynimal when d<n-d} one by one.

\begin{lemma}
Taking a dualization or tensoring with suitable line bundles if necessary, the uniform vector bundle corresponding to case 1, 2, 3, 5 in Proposition \ref{juti ploynimal when d<n-d} are 	
	 $$ \overset{k}{\underset{i=1}{\oplus}} \mathcal{O}_{G(d,n)}{(a_i)},$$ or $$ E\cong H_d^*\bigoplus \mathcal{O}_{G(d,n)}{(j)} $$ for some integers $a_i$ and $j$.
\end{lemma}
\begin{proof}
For case 1, by Proposition 5.2 in \cite{Guy}, $E\cong\overset{k}{\underset{i=1}{\oplus}} \mathcal{O}_{G(d,n)}{(u_i)}$.

For case 2 and 3, by Proposition \ref{polynomial when d<n-d-1}, $S_1(T;X_1,X_2,\cdots X_{d-1};X_d,X_{d+1})$ is a homogeneous polynomial symmetric about $X_1,X_2,\cdots X_{d-1} $ and $ X_d,X_{d+1}$ respectively. So
$c_{HN^1}(T)=T+X_{d}=c_{\mathscr{H}_{H_{d}^{*}}}(T)$. Thus $HN^1$ is a line bundle and $HN^1\cong \mathscr{H}_{H_{d}^{*}}$. Thus $ p^*E $ has $\mathscr{H}_{H_{d}^{*}} $ as a subbundle. By Lemma \ref{1-1lemma}, $E$ has $H_d^*$ as subbundle. So we have following exact sequence for some quotient bundle $\mathcal{O}_{G(d,n)}(i)$:
\begin{center}
	$ 0\longrightarrow H_{d}^{*}  \longrightarrow E \longrightarrow \mathcal{O}_{G(d,n)}(i) \longrightarrow 0 ,$
\end{center}
where $i=0$ or $-1$. Thus $ E\cong H_d^*\bigoplus \mathcal{O}_{G(d,n)}{(i)} $ since $$Ext^1(\mathcal{O}_{G(d,n)}(i), H_d^*)=H^1(G(d,n),H_d^*\otimes \mathcal{O}_{G(d,n)}(-i))=0.$$

For case 5,
\begin{center}
	$ c_{p^*E}(T)=\overset{d}{\underset{i=1}{\prod}}(T+X_{i})(T+2(X_{1}+X_{2}+\cdots+X_{d})) $.
\end{center}

Comparing the Chern polynomial of $ HN^1 $ and $p^{*}\mathcal{O}_{G(d,n)}(2) $, we know that they are isomorphic to each other. So the following exact sequence holds for some quotient bundle $F$.

\begin{center}
	$ 0\longrightarrow p^{*}\mathcal{O}_{G(d,n)}(2) \longrightarrow  p^*E\longrightarrow F\longrightarrow 0 $.

\end{center}

Applying $ p_{*} $ to the short exact sequence, we have the exact sequence

\begin{center}
	$ 0\longrightarrow \mathcal{O}_{G(d,n)}(2)  \longrightarrow E \longrightarrow p_*F \longrightarrow 0 $,
\end{center}
since $ R^1p_{*} \mathcal{O}_{G(d,n)}(2)=0$.
Restricting the exact sequence to a line in $G(d,n)$, we know that $p_*F$ must be a uniform vector bundle of rank $d$. So, by comparing the splitting type and Chern polynomial of $p_*F$, it can be seen that $p_*F$ is isomorphic to $ H_{d}^{*}$  by Proposition 5.4 in \cite{Guy} or Theorem 1.1 in \cite{D-F-G}. Thus $ E\cong  H_{d}^{*} \bigoplus \mathcal{O}_{G(d,n)}(2) $ or direct sum of line bundles because $Ext^1((p_*F)^{*}, \mathcal{O}_{G(d,n)}(2))=0$.

\end{proof}

The following lemma solves for case 4 in Proposition \ref{juti ploynimal when d<n-d}.
\begin{lemma}\label{(1,1,000case)}
	When
	\[ c_{p^*E}(T)=\overset{d}{\underset{i=1}{\prod}}(T+X_{i})(T+(X_{1}+X_{2}+\cdots+X_{d})),\] the vector bundle $E\cong H_d^*\bigoplus \mathcal{O}_{G(d,n)}(1)$.
\end{lemma}
\begin{proof}
The HN-filtration gives the exact sequence
\begin{equation}\label{ex1}
	 0\longrightarrow HN^1\longrightarrow p^*E \longrightarrow F \longrightarrow 0 	
\end{equation}
for some quotient bundle $F$.
Then we have
\[c_{HN^1}(T)=(T+X_{d})(T+(X_{1}+X_{2}+\cdots+X_{d}))\]
and
\begin{eqnarray}
-c_1(HN^1)&=&X_{1}+\cdots+X_{d-1}+2X_{d}\nonumber\\
&=&c_1(H_1)+c_1(H_2/H_1)+\cdots+c_1(H_{d-1}/H_{d-2})+2c_1(H_d/H_{d-1})\nonumber\\
&=&-c_1(H_{d-1})+2c_1(H_d).\label{c1r}
\end{eqnarray}

Restricting the HN-filtration to the fiber $pr_{1}^{-1}(y)$ at a point $y\in F(d-1,d,n)$, we have
\begin{center}
	$ 0\longrightarrow HN^1|_{pr_{1}^{-1}(y)}\longrightarrow p^*E|_{pr_{1}^{-1}(y)} \longrightarrow F|_{pr_{1}^{-1}(y)} \longrightarrow 0, $
	\end{center}
and $c_1(HN^1|_{pr_{1}^{-1}(y)})=0$ by equation (\ref{c1r}). Notice that  $pr_{1}^{-1}(y)$ is a subvariety of $p^{-1}(q(y))$, so $p^*E|_{pr_{1}^{-1}(y)}$ is a trivial bundle. Thus $HN^1|_{pr_{1}^{-1}(y)}$ is a trivial bundle and $pr_{1*}HN^1$ is a $2$-bundle over $F(d-1,d,n)$.

Applying $pr_{1*}$ to the exact sequence (\ref{ex1}), we have
\begin{center}
	$ 0\longrightarrow pr_{1*}HN^1\longrightarrow q^*E \longrightarrow pr_{1*}F \longrightarrow 0.$	
\end{center}

Since the canonical morphism from $pr^*_1(pr_{1*}HN^1)$ to $HN^1$ restrict to any fiber of $pr_1$ is an isomorphism as two trivial bundles of rank $2$, we can get $pr^*_1(pr_{1*}HN^1)\cong HN^1$. So $ c_1(HN^1)=pr^*_{1}c_1(pr_{1*}HN^1) $. By equation (\ref{c1r}), $$c_1(HN^1)=c_1(H_{d-1})-2c_1(H_d)=pr_{1}^{*}(c_1(\overset{\sim}{H_{d-1}})-2c_1(\overset{\sim}{H_d})), $$
where $\overset{\sim}{H_{d-1}}$ and $\overset{\sim}{H_d}$ are universal subbundles over $F(d-1,d,n)$. Thus $$c_1(pr_{1*}HN^1)=c_1(\overset{\sim}{H_{d-1}})-2c_1(\overset{\sim}{H_d}).$$
So, when restrict $c_1(pr_{1*}HN^1)$ to the fiber of $q$ at $x\in G(d, n)$, we have
$$c_1(pr_{1*}HN^1)|_{q^{-1}(x)}=\mathcal{O}_{q^{-1}(x)}(-1).$$
Thus $pr_{1*}HN^1|_{q^{-1}(x)}=\mathcal{O}_{q^{-1}(x)}\oplus \mathcal{O}_{q^{-1}(x)}(-1)$ since $pr_{1*}HN^1|_{q^{-1}(x)}$ is a subbundle of $q^*E|_{q^{-1}(x)}$ which is a trivial bundle. Therefore, $q_*(pr_{1*}HN^1)$ is a line bundle over $G(d,n)$, i.e. $q_*(pr_{1*}HN^1)=\mathcal{O}_{G(d,n)}(i) $ for some $i$.
So we have exact sequence
\begin{center}
	$ 0\longrightarrow \mathcal{O}_{G(d,n)}(i)\longrightarrow E \longrightarrow p_*F \longrightarrow 0. $
\end{center}
Then  \[c_{p^*E}(T)=c_{p^*(p_*F)}(T)(T+i(X_{1}+X_{2}+\cdots+X_{d}))=\overset{d}{\underset{i=1}{\prod}}(T+X_{i})(T+(X_{1}+X_{2}+\cdots+X_{d})).\] Clearly, $i=1$, so we deduce that
\begin{center}
	$ 0\longrightarrow \mathcal{O}_{G(d,n)}(1)\longrightarrow E \longrightarrow p_*F  \longrightarrow 0.$
\end{center}

Obviously, $p_*F$ is a uniform vector bundle of rank $d$ over $G(d,n)$ with split type $(1,0,\cdots,0)$. By Proposition 5.4 in \cite{Guy} or Theorem 1.1 in \cite{D-F-G}, $p_*F$ is isomorphic to $H_{d}^{*}$.

Thus $ E\cong H_d^*\bigoplus \mathcal{O}_{G(d,n)}(1)$ since $Ext^1(H_d^*, \mathcal{O}_{G(d,n)}(1))=0$.
\end{proof}

The following lemma solves for case 6 in Proposition \ref{juti ploynimal when d<n-d}.
\begin{lemma}\label{S^2 case}
 There does not exist uniform vector bundle of rank $d+1$ whose Chern polynomial is
		$$ c_{p^*E}(T) =\overset{d}{\underset{i=1}{\prod}}(T+2X_{i})(T+(X_{1}+X_{2}+\cdots+X_{d})) $$ when $d\geq3$.
		
		If $d=2$ and $ c_{p^*E}(T) =\overset{2}{\underset{i=1}{\prod}}(T+2X_{i})(T+(X_{1}+X_{2})) $, then $E\cong S^{2}H_{2} $.
	
\end{lemma}
\begin{proof}
In this case, we have $$c_{HN^1}(T)=T+2X_{d}=c_{ \mathscr{H}_{H_{d}^{*}}^{\otimes 2}}(T)$$ and $$c_{HN^{2}/HN^{1}}(T)=T+X_{1}+X_{2}+\cdots+X_{d}=c_{p^{*}\mathcal{O}_{G(d,n)}(1)}(T).$$
Since $HN^1$, $\mathscr{H}_{H_{d}^{*}}^{\otimes 2}$, $HN^{2}/HN^{1}$ and $p^{*}\mathcal{O}_{G(d,n)}(1)$ are all vector bundle of rank $1$, we get that
$HN^1\cong  \mathscr{H}_{H_{d}^{*}}^{\otimes 2}$ and $HN^{2}/HN^{1}\cong p^{*}\mathcal{O}_{G(d,n)}(1)$.
 The HN-filtration gives
\begin{equation}\label{HN1}
0\longrightarrow   \mathscr{H}_{H_{d}^{*}}^{\otimes 2} \longrightarrow HN^2 \longrightarrow p^{*}\mathcal{O}_{G(d,n)}(1) \longrightarrow 0
\end{equation}
and 	
\begin{equation}\label{HN2}
0\longrightarrow  HN^2 \longrightarrow p^*E \longrightarrow F\longrightarrow 0
\end{equation}
 for some quotient bundle $F$.

When $d\ge 3$, viewing $F(d-1,d,n)$ as a projective bundle over $G(d,n)$, we get $$R^{i}q_{*}(\mathcal{O}_{F(d-1,d,n)}(-2))=0$$ for $i>0$ (see \cite{Har} Chapter III, Exercise 8.4 (a)).
Since $\mathscr{H}_{H_{d}^{*}}^{\otimes 2}\otimes p^{*}\mathcal{O}_{G(d,n)}(-1)$ restricts to a fiber of $pr_1$ is trivial for any $d\ge 2$, we get
\begin{equation}\label{Ri=0}
R^{i}pr_{1*}(\mathscr{H}_{H_{d}^{*}}^{\otimes 2}\otimes p^{*}\mathcal{O}_{G(d,n)}(-1))=0
\end{equation} for $i>0$.
So
\begin{eqnarray*}
&&H^{1}(F(d-1,d,d+1,n),\mathscr{H}_{H_{d}^{*}}^{\otimes 2}\otimes p^{*}\mathcal{O}_{G(d,n)}(-1))\\
&\cong &H^1(F(d-1,d,n),pr_{1*}(\mathscr{H}_{H_{d}^{*}}^{\otimes 2}\otimes p^{*}\mathcal{O}_{G(d,n)}(-1)))\\
&=&H^1(F(d-1,d,n),\mathcal{O}_{F(d-1,d,n)}(-2)\otimes q^{*}\mathcal{O}_{G(d,n)}(-1))\\
&=&H^1(G(d,n),q_{*}\mathcal{O}_{F(d-1,d,n)}(-2)\otimes \mathcal{O}_{G(d,n)}(-1))\\
&=&H^1(G(d,n),0)\\
&=&0.
\end{eqnarray*}
Then $Ext^{1}(p^{*}\mathcal{O}_{G(d,n)}(1), \mathscr{H}_{H_{d}^{*}}^{\otimes 2})=0$.
So, from (\ref{HN1}), \[ HN^2\simeq \mathscr{H}_{H_{d}^{*}}^{\otimes 2}\oplus p^{*}\mathcal{O}_{G(d,n)}(1).\]

Applying $p_*$ to (\ref{HN2}) and noticing that $p_*(\mathscr{H}_{H_{d}^{*}}^{\otimes 2}))=q_*(\mathcal{O}_{F(d-1,d,n)}(-2))=0$ and $R^1p_*(\mathscr{H}_{H_{d}^{*}}^{\otimes 2}))=0$, we have
\[0\longrightarrow \mathcal{O}_{G(d,n)}(1)  \longrightarrow E \longrightarrow p_{*}F\longrightarrow 0.\]
Moreover, over $ F(d-1,d,d+1,n) $, we have the commutative diagram

\begin{center}
	$\xymatrix{
		0\ar[r] & p^{*}\mathcal{O}_{G(d,n)}(1) \ar[r] \ar[d] & p^{*}E \ar[d]^{id} \ar[r] & p^{*}p_{*}F\ar[r]\ar[d]&0\\
		0\ar[r] & HN^2 \ar[r]&p^{*}E \ar[r]&F\ar[r]&0.\\
	}$	
\end{center}

The snake lemma gives the  exact sequence

\begin{center}
	$ 0\longrightarrow  \mathscr{H}_{H_{d}^{*}}^{\otimes 2} \longrightarrow p^{*}p_{*}F \longrightarrow F \longrightarrow 0 $.
\end{center}

Restricting to the fiber of $s$ at some point $l$ in $F(d-1,d+1,n)$, which is isomorphic to $\mathbb{P}^1$, by Proposition 2.3 in \cite{Guy}, we have
\[T_{F(d-1,d,d+1,n)/G(d,n)}|_{s^{-1}(l)}=\mathcal{O}_{s^{-1}(l)}(-1)^{\oplus m}.
\]
Moreover, \[\mathscr{H}_{H_{d}^{*}}|_{s^{-1}(l)}=\mathcal{O}_{s^{-1}(l)}(1)\] and
\[F|_{s^{-1}(l)}\cong \mathcal{O}_{s^{-1}(l)}^{\oplus d-1}\] since $F=p^{*}E/HN^2$.
So
\[Hom(T_{F(d-1,d,d+1,n)/G(d,n)},\mathscr{H}om(\mathscr{H}_{H_{d}^{*}}^{\otimes 2},F))|_{s^{-1}(l)}=0\]
\[\Rightarrow Hom(T_{F(d-1,d,d+1,n)/G(d,n)},\mathscr{H}om(\mathscr{H}_{H_{d}^{*}}^{\otimes 2},F))=0.\]

Thus, by Descente-Lemma (Lemma 2.1.2 in \cite{O-S-S}), there is a line subbundle $L$ of $p_{*}F$ such that $p^{*}L $ isomorphic to $\mathscr{H}_{H_{d}^{*}}^{\otimes 2}$, which is impossible because $p_*p^{*}L\cong L$ while $p_*\mathscr{H}_{H_{d}^{*}}^{\otimes 2}=0$.

When $d=2$, from the HN-filtration we have following two exact sequences
\begin{equation}\label{HN1d2}
 0\longrightarrow  \mathscr{H}_{H_{2}^{*}}^{\otimes 2} \longrightarrow HN^2 \longrightarrow p^{*}\mathcal{O}_{G(2,n)}(1) \longrightarrow 0
\end{equation}
and
\begin{equation}\label{HN2d2}
 0\longrightarrow  HN^2 \longrightarrow p^*E \longrightarrow F \longrightarrow 0.
\end{equation}	
Now $F$ is a line bundle, so comparing the Chern polynomial we can get
\[F=\mathscr{H}_{H_{2}^*}^{\otimes -2}\otimes p^{*}\mathcal{O}_{G(2,n)}(2).\]
Next, we want to consider the extension of (\ref{HN1d2}) by calculating $H^{1}(F(1,2,3,n), \mathscr{H}_{H_{2}^{*}}^{\otimes 2}\otimes p^{*}\mathcal{O}_{G(2,n)}(-1))$.

By Leray spectral sequence, we have exact sequence
\begin{multline*}
H^{1}(G(2,n),p_{*}(\mathscr{H}_{H_{2}^{*}}^{\otimes 2}\otimes p^{*}\mathcal{O}_{G(2,n)}(-1)))\rightarrow H^{1}(F(1,2,3,n), \mathscr{H}_{H_{2}^{*}}^{\otimes 2}\otimes p^{*}\mathcal{O}_{G(2,n)}(-1))\\
\rightarrow H^{0}(G(2,n),R^{1}p_{*}(\mathscr{H}_{H_{2}^{*}}^{\otimes 2}\otimes p^{*}\mathcal{O}_{G(2,n)}(-1)))\rightarrow H^{2}(G(2,n),p_{*}(\mathscr{H}_{H_{2}^{*}}^{\otimes 2}\otimes p^{*}\mathcal{O}_{G(2,n)}(-1))).
\end{multline*}
Since $p_{*}\mathscr{H}_{H_{2}^{*}}^{\otimes 2}=0$ and $R^1p_{*}\mathscr{H}_{H_{2}^{*}}^{\otimes 2}=R^1q_*(\mathcal{O}_{\mathbb{P}(H_2)}(-2))=q_*\mathcal{O}_{\mathbb{P}(H_2)}\otimes \wedge^2 H_2^*=\mathcal{O}_{G(2,n)}(1)$ (see \cite{Har} Chapter III, Exercise 8.4 (c)), we have
\begin{eqnarray*}
&&H^{1}(F(1,2,3,n), \mathscr{H}_{H_{2}^{*}}^{\otimes 2}\otimes p^{*}\mathcal{O}_{G(2,n)}(-1))\\
&=&H^{0}(G(2,n),R^{1}p_{*}(\mathscr{H}_{H_{2}^{*}}^{\otimes 2}\otimes p^{*}\mathcal{O}_{G(2,n)}(-1)))\\
&=&H^{0}(G(2,n),R^{1}p_{*}\mathscr{H}_{H_{2}^{*}}^{\otimes 2}\otimes \mathcal{O}_{G(2,n)}(-1))\\
&=&H^{0}(G(2,n),\mathcal{O}_{G(2,n)}).
\end{eqnarray*}

Thus, \[\text{dim} Ext^1(p^{*}\mathcal{O}_{G(2,n)}(-1)), \mathscr{H}_{H_{2}^{*}}^{\otimes 2})=h^{0}(G(d,n),\mathcal{O}_{G(2,n)})=1.\]

From the same argument above, $HN^{2}$ is not the direct sum of $p^{*}\mathcal{O}_{G(2,n)}(1))$ and $\mathscr{H}_{H_{2}^{*}}^{\otimes 2}$.
So the extension (\ref{HN1d2}) is nontrivial.

Applying $p_{*}$ to (\ref{HN2d2}), we have
\begin{equation}\label{S2H}
0\longrightarrow p_{*}HN^{2}\longrightarrow E\longrightarrow S^{2}H_{2}(2)\longrightarrow R^{1}p_{*}HN^{2}\longrightarrow 0.
\end{equation}

Applying $pr_{1*}$ to (\ref{HN1d2}) and setting $M=pr_{1*}HN^{2}$, we have
\begin{equation}\label{M}
0\longrightarrow \mathcal{O}_{\mathbb{P}(H_2)}(-2)\longrightarrow M \longrightarrow q^{*}\mathcal{O}_{G(2,n)}(1)\longrightarrow 0.
\end{equation}
If the above exact sequence (\ref{M}) splits, then
\[ pr^*_1pr_{1*}HN^2=pr^*_1M\simeq \mathscr{H}_{H_{d}^{*}}^{\otimes 2}\oplus p^{*}\mathcal{O}_{G(d,n)}(1).\]
Since the canonical morphism from $pr^*_1(pr_{1*}HN^2)$ to $HN^2$ restrict to any fiber of $pr_1$ is an isomorphism as two trivial bundles of rank $2$, we get $pr^*_1(pr_{1*}HN^2)\cong HN^2$.
So \[HN^2\simeq \mathscr{H}_{H_{d}^{*}}^{\otimes 2}\oplus p^{*}\mathcal{O}_{G(d,n)}(1),\] which is a contradiction.
So (\ref{M}) is a nontrivial extension.

Next, we are going to prove $q_{*}M=p_{*}HN^2=0$.

When $d=2$, $q^{-1}(x)$ is a line for some point $x\in G(2,n)$. We have standard exact sequence
\[0\longrightarrow \mathcal{I}_{q^{-1}(x)} \longrightarrow \mathcal{O}_{\mathbb{P}(H_2)} \stackrel{h_{x}}{\longrightarrow} \mathcal{O}_{q^{-1}(x)}\longrightarrow0.\]

Tensoring with $\mathcal{O}_{\mathbb{P}(H_2)}(-2)\otimes q^{*}\mathcal{O}_{G(d,n)}(-1)$, we get the short exact sequence
\begin{multline*}
0\longrightarrow \mathcal{I}_{q^{-1}(x)}\otimes \mathcal{O}_{\mathbb{P}(H_2)}(-2)\otimes q^{*}\mathcal{O}_{G(2,n)}(-1) \longrightarrow\\ \mathcal{O}_{\mathbb{P}(H_2)}(-2)\otimes q^{*}\mathcal{O}_{G(2,n)}(-1) \stackrel{h_{x}}{\longrightarrow} \mathcal{O}_{q^{-1}(x)}(-2)\longrightarrow0,
\end{multline*}
which induces the long exact sequence
\begin{multline*}
0=H^0(q^{-1}(x),\mathcal{O}_{q^{-1}(x)}(-2))\longrightarrow H^1(\mathbb{P}(H_2),\mathcal{I}_{q^{-1}(x)}\otimes \mathcal{O}_{\mathbb{P}(H_2)}(-2)\otimes q^{*}\mathcal{O}_{G(2,n)}(-1)) \longrightarrow\\
 H^1(\mathbb{P}(H_2),\mathcal{O}_{\mathbb{P}(H_2)}(-2)\otimes q^{*}\mathcal{O}_{G(2,n)}(-1))\stackrel{f_x}{\longrightarrow} H^1(q^{-1}(x),\mathcal{O}_{q^{-1}(x)}(-2)).
\end{multline*}

Over $F(1,2,n)=\mathbb{P}(H_2)$, we have the communitative diagram

\begin{displaymath}
	\xymatrix{
		0\ar[r] & \mathcal{O}_{\mathbb{P}(H_2)}(-2)\otimes q^{*}\mathcal{O}_{G(2,n)}(-1) \ar[r] \ar[d] &M\otimes q^{*}\mathcal{O}_{G(2,n)}(-1) \ar[d]^{id} \ar[r] & O_{\mathbb{P}(H_2)}\ar[r]\ar[d]&0\\
		0\ar[r] & \mathcal{O}_{q^{-1}(x)}(-2) \ar[r]&M|_{q^{-1}(x)}\ar[r]&\mathcal{O}_{q^{-1}(x)}\ar[r]&0,
	}
\end{displaymath}
which induces the commutative diagram
\begin{displaymath}
\xymatrix{
H^0(\mathbb{P}(H_2),\mathcal{O}_{\mathbb{P}(H_2)})\ar[r]^(0.35){\delta_M}\ar[d]^{h'_x}&H^1(\mathbb{P}(H_2),\mathcal{O}_{\mathbb{P}(H_2)}(-2)\otimes q^{*}\mathcal{O}_{G(2,n)}(-1))\ar[d]^{f_x}\\		H^0(q^{-1}(x),\mathcal{O}_{q^{-1}(x)})\ar[r]^(0.45){\delta_{M|{q^{-1}(x)}}}&H^1(q^{-1}(x),\mathcal{O}_{q^{-1}(x)}(-2)).
	}
\end{displaymath}
Then
\[f_{x}\circ {\delta_M(1)}=\delta_{M|{q^{-1}(x)}}\circ h'_x(1)\]
and
	\[h'_x(1)=1, ~\text{and}~ \delta_M(1)=t\neq 0\]
since (\ref{M}) is the nontrivial extension.

If (\ref{M}) restricting to the fiber $q^{-1}(x)$ is a trivial extension, then $\delta_{M|{q^{-1}(x)}}(1)=f_x(t)=0$.
So \[h^1(\mathbb{P}(H_2),\mathcal{I}_{q^{-1}(x)}\otimes \mathcal{O}_{\mathbb{P}(H_2)}(-2)\otimes q^{*}\mathcal{O}_{G(2,n)}(-1))=h^1(\mathbb{P}(H_2),\mathcal{O}_{\mathbb{P}(H_2)}(-2)\otimes q^{*}\mathcal{O}_{G(2,n)}(-1)).\]

For any other point $y\in G(2,n)$, since the Grassmannian is rationally connected, we have $\mathcal{I}_{q^{-1}(x)}\cong \mathcal{I}_{q^{-1}(y)}$. So
\[h^1(\mathbb{P}(H_2),\mathcal{I}_{q^{-1}(y)}\otimes \mathcal{O}_{\mathbb{P}(H_2)}(-2)\otimes q^{*}\mathcal{O}_{G(2,n)}(-1))=h^1(\mathbb{P}(H_2),\mathcal{O}_{\mathbb{P}(H_2)}(-2)\otimes q^{*}\mathcal{O}_{G(2,n)}(-1)),\] for all $y \in G(2,n),$  i.e. $f_y=0$.
It follows that
\[M|_{q^{-1}(x)}\cong \mathcal{O}_{\mathbb{P}^1}(-2)\oplus \mathcal{O}_{\mathbb{P}^1},~\forall x \in G(2,n).\]
Thus $q_{*}M$ is a subbundle of of $E$ of rank $1$.
We get the following exact sequence for some quotient bundle $K$
\[0\rightarrow q_{*}M\longrightarrow E \longrightarrow K\longrightarrow 0. \]
Comparing the Chern polynomial of $p^*E$ and $p^*q_{*}M$, we have $p^*q_{*}M \cong p^*\mathcal{O}_{G(2,n)}(1)$ and the following commutative diagram over $F(1,2,3,n)$ holds by checking on any $p$-fibers:
\begin{displaymath}
\xymatrix{
		0\ar[r] & p^{*}\mathcal{O}_{G(d,n)}(1) \ar[r] \ar[d] & p^{*}E \ar[d]^{id} \ar[r] & p^{*}K\ar[r]\ar[d]&0\\
		0\ar[r] & HN^2 \ar[r]&p^{*}E \ar[r]&F\ar[r]&0.\\
	}	
\end{displaymath}
Assume that the sequence
\[0\longrightarrow p^{*}\mathcal{O}_{G(d,n)}(1)\longrightarrow HN^2\longrightarrow D\longrightarrow0\]
is exact for some quotient bundle $D$.

Comparing the Chern polynomial of $D$ and $\mathscr{H}_{H_{d}^{*}}^{\otimes 2} $, since $c_{D}(T)=T+2X_2=c_{ \mathscr{H}_{H_{d}^{*}}^{\otimes 2} }(T)$, we have $D\cong \mathscr{H}_{H_{d}^{*}}^{\otimes 2} $.
Then the snake lemma gives the exact sequence
\[0\longrightarrow  \mathscr{H}_{H_{d}^{*}}^{\otimes 2} \longrightarrow p^{*}K \longrightarrow F \longrightarrow 0.\]

With the similar argument above, there is a rank $1$ vector subbundle $L$ of $K$ such that $p^{*}L $ isomorphic to $\mathscr{H}_{H_{d}^{*}}^{\otimes 2}$, which is impossible because $p_*p^{*}L\cong L$ while $p_*\mathscr{H}_{H_{d}^{*}}^{\otimes 2}=0$.
Therefore,  $f_x\neq 0$, i.e $\delta_{M|{q^{-1}(x)}}\neq 0$, which means that $M|_{q^{-1}(x)}\cong \mathcal{O}_{\mathbb{P}^1}(-1)\oplus \mathcal{O}_{\mathbb{P}^1}(-1)$ for all $q$ fibers.

Since $p=q\circ pr_{1*}$ and $M=pr_{1*}HN^{2}$ has split type $(-1,-1)$ at each fiber of $q$ , we have $p_{*}HN^{2}=R^{1}p_{*}HN^{2}=0$. Thus $E\cong S^{2}H_{2}(2)$ from (\ref{S2H}).
\end{proof}

\begin{proposition}
The uniform vector bundles correspond Proposition \ref{juti ploynimal when d<n-d} are one of the following:
\begin{center}
	$H_d{(a)}\bigoplus \mathcal{O}_{G(d,n)}{(b)},~ H_d^*{(c)}\bigoplus \mathcal{O}_{G(d,n)}{(d)},~\overset{k}{\underset{i=1}{\oplus}} O^{\oplus r_{i}}_{G(d,n)}{(e_i)}~ \text{and} ~S^2 H_2(f)$
	\end{center}	
where $a,b,c, d, e_i, f \in \mathbb{Z}$ and $H_d$ is the tautological subbundle on  $G(d,n)$.
\end{proposition}

\section{Uniform $(d+1)$-bundle over $G(d,n)$ when $d=n-d-1$ and $d=n-d$}\label{section n-d}
In characteristic $0$, the vector bundle $E$ with split form $(k;u_1,\cdots,u_k;r_1,\cdots,r_k)$ ($u_1>u_2>\cdots>u_k$) is an extension of two uniform vector bundles if $u_{i-1}-u_i\geq 2$. So we may assume $u_i$s' are consecutive in this section.

\subsection{Uniform $(d+1)$-bundle over $G(d,n)$ when $d=n-d-1$}
 When the uniform vector bundle $E$ is rank of $d+1$, under the assumption $d+1=n-d$, by Theorem \ref{Chow}, the Chern polynomial of $p^{*}E$ is
\begin{center}
	$c_{p^{*}E}(T)=\overset{k}{\underset{i=1}{\prod}}c_{s^{*}(E_i)\otimes \mathcal{O}_{s}(u_i)}(T)+a\sum_{n-d}(X_1,\cdots,X_{d+1}), $
\end{center}
where $ \sum_{n-d}(X_1,\cdots,X_{d+1})=\underset{a_1+a_2+\cdots+a_{d+1}=n-d}{\sum}X_1^{a_1}X_2^{a_2}\cdots X_{d+1}^{a_{d+1}} $.

With the help of \cite{E-H-S}, only the following cases may happen:
\begin{itemize}
\item[(A)] $a=0$;
\item[(B)] $k=1$;
\item[(C)] $k=2$.
\end{itemize}

When $a=0$, classifying of uniform $d+1$-bundle has been solved by same arguments in the above section. For cases (B) and (C),
the uniform bundle $E$ is direct sum of line bundles, $Q_{n-d}(s)$ or $Q_{n-d}^*(t)$, where $s, t\in \mathbb{Z}$ (see \cite{Guy} Proposition 5.4).

\subsection{Chern polynomial of $p^{*}E$ when $d=n-d$}
So we only need to consider $d=n-d$. Now, we suppose that the Chern polynomial of $p^{*}E$ is
\begin{multline*}
	c_{p^{*}E}(T)=\overset{k}{\underset{i=1}{\prod}}c_{s^{*}(E_i)\otimes \mathcal{O}_{s}(u_i)}(T)+(aT+b_{1}X_{1}+\cdots\\+b_{d-1}X_{d-1}+cX_{d}+eX_{d+1})\sum_{n-d}(X_1,\cdots,X_{d+1})+f\sum_{n-d+1}(X_1,\cdots,X_{d+1}).
\end{multline*}
Since $ c_{p^{*}E}(T)$, $c_{s^{*}(E_i)\otimes \mathcal{O}_{s}(u_i)}(T)$ and $\sum_{m}(X_1,\cdots,X_{d+1})~ (m=n-d, n-d+1)$ are symmetric about $X_1,\cdots, X_{d-1}$ , we see that $b_1=b_2=\cdots=b_{d-1}$. So we rewrite the Chern polynomial of $p^{*}E$ as
\begin{multline*}
	c_{p^{*}E}(T)=\overset{k}{\underset{i=1}{\prod}}c_{s^{*}(E_i)\otimes \mathcal{O}_{s}(u_i)}(T)+(aT+b(X_{1}+\cdots\\+X_{d-1})+cX_{d}+eX_{d+1})\sum_{n-d}(X_1,\cdots,X_{d+1})+f\sum_{n-d+1}(X_1,\cdots,X_{d+1}).
\end{multline*}

Now we study the behavior of general polynomial
\begin{multline*}
(\mathscr{E}): ~E(T;X_1,X_2,\cdots X_d)=\overset{k}{\underset{i=1}{\prod}}S_i(T+u_i X_d;X_1,X_2,\cdots X_{d-1};X_d,X_{d+1})+\\(aT+b(X_{1}+\cdots+X_{d-1})+cX_{d}+eX_{d+1})\sum_{n-d}(X_1,\cdots,X_{d+1})+f\sum_{n-d+1}(X_1,\cdots,X_{d+1}),
\end{multline*}
where $E(T;X_1,X_2,\cdots X_d)$ is a homogeneous polynomial symmetric about $X_1,X_2,\cdots X_d$ and
$S_i(T;X_1,X_2,\cdots X_{d-1};X_d,X_{d+1})$ is a homogeneous polynomial of degree $r_i$ and symmetric about $X_1,X_2,\cdots X_{d-1} $ and $ X_d,X_{d+1}$ respectively.

Replacing $T$ by $T+u_{j}(X_{1}+\cdots+X_{d})$, we have
\begin{multline*}
 (\mathscr{E}_j):~ E^{j}(T;X_1,X_2,\cdots X_d)=\overset{k}{\underset{i=1}{\prod}}S_i^{j}(T+(u_{i}-u_{j})X_d;X_1,X_2,\cdots X_{d-1};X_d,X_{d+1})+\\(aT+b^{j}(X_{1}+\cdots+X_{d-1})+c^{j}X_{d}+eX_{d+1})\sum_{n-d}(X_1,\cdots,X_{d+1})+f\sum_{n-d+1}(X_1,\cdots,X_{d+1}).	
\end{multline*}

\textbf{Case $k=1$:}
In this case $r_1=d+1$. Without loss of generality, we can assume $u_1=0$. We use the result in \cite{D-F-G} (Proposition 3.1) to get that $E$ a is trivial bundle.

\textbf{Case $k\geq 2$:}
\begin{proposition}
	$e+f=0$.
\end{proposition}
\begin{proof}
Taking $X_1=X_2=\cdots=X_{d-1}=0$ in $ \mathscr{E}(j)$ and
setting $$E^{j}(T;X_d)=T\overset{\sim}{E^{j}}(T;X_d)+h^{i}X_{d+1}^{n-d+1},$$ we get
\begin{multline*}
	T\overset{\sim}{E^{j}}(T;X_d)=\overset{k}{\underset{i=1}{\prod}}S_i^{j}(T+(u_{i}-u_{j})X_d, X_d, X_{d+1})\\+(aT+c^{j}X_{d}+(e+f)X_{d+1})\sum_{n-d}(X_d,X_{d+1})+f^{i}X_{d+1}^{n-d+1}.
\end{multline*}

By the arguments after Proposition 1 in \cite{Bal} ($c$ is $e+f$ here and $n$ is $d$ here, $x_{i}$ is $f^{i}$ here), when $k\geq 4$ or $d$ is odd, there must exist an index $i$ such that $f^i=0$. By Lemma 1 in \cite{Bal}, we have $e+f=0$.
So we only need to consider $k\leq 3$ and $d$ is even.

If $e+f\neq 0$, by the arguments after Proposition 1 in \cite{Bal}, we get $k\geq 3$. So $k=3$ and $d$ is even.
When $k=3$,  by the arguments after Proposition 1 in \cite{Bal}, we have $r_1=r_3$. Since $d$ is even, if $r_1$ is odd, by [\cite{Ell} Lemma V.3.1], $e+f=0$ which is a contradiction. If $r_1\geq 4$, by [\cite{Ell} V.6.3.1], $e+f=0$, a contradiction. When $r_1=2$, by [\cite{Ell} V.6.4.2 case(2)], $e+f=0$, also a contradiction.
\end{proof}

%
%
%
%
Comparing the degree of $X_{d+1}$ in both sides of $\mathscr{E}$, there must exist an index $j_0$ such that the coefficient of $X_{d+1}^{r_{j_0}}$ in $S_{j_0}(T+u_{j_0} X_d;X_1,X_2,\cdots X_{d-1};X_d,X_{d+1})$ is zero, i.e,  the coefficient of $X_{d+1}^{r_{j_0}}$ in $S_{j_0}^{j_0}(T;X_1,X_2,\cdots X_{d-1};X_d,X_{d+1})$ is zero.

\begin{proposition}\label{form of Si}
	If $a=0$ in the equation $ \mathscr{E}(j_0) $, then $b^{j_0}=c^{j_0}=0$.
\end{proposition}
\begin{proof}
Compare the coefficient of $X_{d+1}^{d}$ in both sides of  $ \mathscr{E}(j_0) $.  Since the coefficient of $X_{d+1}^{r_{j_0}}$ in $S_{j_0}^{j_0}(T;X_1,X_2,\cdots X_{d-1};X_d,X_{d+1})$ is zero, if one of $b^{j_0}$ and $c^{j_0}$ is nonzero, then the coefficient of $X_{d+1}^{r_{i}}$  in $S_i^{j_0}(T+(u_{i}-u_{j_0}) X_d;X_1,X_2,\cdots X_{d-1};X_d,X_{d+1})$ is nonzero constant when $i\neq j_0$.

So, when $i\neq j_0$,
$S_i^{j_0}(T+(u_{i}-u_{j_0}) X_d,X_1,\cdots,X_{d+1})$ is of form $a_{i}X_{d+1}^{r_i}+\text{lower degree terms}$, where $a_{i}\neq0$ and
$S_{j_0}^{j_0}(T,X_1,\cdots,X_{d+1})$ is the form of $F_{j_0}^1(T,X_1,\cdots,X_d)X_{d+1}^{r_{j_0}-1}+\text{lower degree terms}$.

Taking $X_1=\cdots=X_{d}=0$ in $ \mathscr{E}(j_0) $, we have
\begin{center}
	$T^{d+1}=S_{j_0}^{j_0}(T,X_{d+1})\overset{k}{\underset{i\neq j_0}{\prod}}(a_{i}X_{d+1}^{r_i}+\text{lower degree terms})$,
\end{center}
which is impossible. So $b^{j_0}=c^{j_0}=0$.
\end{proof}

\begin{proposition}\label{polyan}
If $a\neq 0$ in the equation
\begin{multline*}
		\mathscr{E}(j_0): E^{j_0}(T;X_1,X_2,\cdots X_d)=\overset{k}{\underset{i=1}{\prod}}S_i^{j_0}(T+(u_{i}-u_{j}) X_d;X_1,X_2,\cdots X_{d-1};X_d,X_{d+1})+\\a(T+m_{1}(X_{1}+\cdots+X_{d-1})+m_{2}X_{d})\sum_{n-d}(X_1,\cdots,X_{d+1})+f\sum_{n-d+1}(X_1,\cdots,X_{d})
	\end{multline*}
where $m_1=\frac{b^{j_0}}{a}, m_2=\frac{c^{j_0}}{a}$, then
\[T+m_{1}(X_{1}+\cdots+X_{d-1})+m_{2}X_{d}|S_{j_0}^{j_0}(T;X_1,X_2,\cdots X_{d-1};X_d,X_{d+1})\] and
\[T+m_{1}(X_{1}+\cdots+X_{d-1})+m_{2}X_{d}|E^{j_0}(T;X_1,X_2,\cdots X_d)-f\sum_{n-d+1}(X_1,\cdots,X_{d}).\]
Moreover, if deg $S_{j_0}^{j_0}=1$, then $m_2=0$; if deg $S_{j_0}^{j_0}>1$, then $m_1=m_2=0$.
\end{proposition}
\begin{proof}
Suppose that $a\neq 0$. By the similar argument in Proposition \ref{form of Si},
\[S_i^{j_0}(T+(u_{i}-u_{j_0}) X_d,X_1,\cdots,X_{d+1})\] is of the form  $a_{i}X_{d+1}^{r_i}+\text{lower degree terms}$, where $a_{i}\neq0$, when $i\neq j_0$,
and
$S_{j_0}^{j_0}(T,X_1,\cdots,X_{d+1})$ is of the form $F_{j_0}^1(T,X_1,\cdots,X_d)X_{d+1}^{r_{j_0}-1}+\text{lower degree terms}$.

Taking $T=-m_{1}(X_{1}+\cdots+X_{d-1})-m_{2}X_{d}$ in  $ \mathscr{E}(j_0) $, we get that
\begin{multline*}
	E^{j_0}(-m_{1}(X_{1}+\cdots+X_{d-1})-m_{2}X_{d};X_1,X_2,\cdots X_d)-f\sum_{n-d+1}(X_1,\cdots,X_{d})=\\\overset{k}{\underset{i=1}{\prod}}S_i^{j_0}(-m_{1}(X_{1}+\cdots+X_{d-1})-m_{2}X_{d}+(u_{i}-u_{j_0}) X_d;X_1,X_2,\cdots X_{d-1};X_d,X_{d+1}).
\end{multline*}
Since left side of the above equation is independent of $X_{d+1}$ and $a_{i}\neq0$,
\begin{center}
	$S_{j_0}^{j_0}(-m_{1}(X_{1}+\cdots+X_{d-1})-m_{2}X_{d};X_1,X_2,\cdots X_{d-1};X_d,X_{d+1})=0$.
\end{center}
So \[T+m_{1}(X_{1}+\cdots+X_{d-1})+m_{2}X_{d}|S_{j_0}^{j_0}(T;X_1,X_2,\cdots X_{d-1};X_d,X_{d+1})
\] and
\[T+m_{1}(X_{1}+\cdots+X_{d-1})+m_{2}X_{d}|E^{j_0}(T;X_1,X_2,\cdots X_d)-f\sum_{n-d+1}(X_1,\cdots,X_{d}).\]
	
Assume that
\[E^{j_0}(T;X_1,\cdots,X_d)=(T+m_1(X_1+\cdots+X_{d-1})+m_2X_d){\overset{\sim}{E}}^{j_0}(T,X_1,\cdots,X_d)\]
and
\begin{multline*}
S_{j_0}^{j_0}(T,X_1\cdots,X_{d-1};X_d,X_{d+1})\\=(T+m_1(X_1+\cdots+X_{d-1})+m_2X_d)\overset{\sim}{S}_{j_0}^{j_0}(T,X_1,\cdots,X_{d-1};X_d,X_{d+1}).
\end{multline*}

When deg $S_{j_0}^{j_0}=1$, since $S_{j_0}^{j_0}(T;X_1,X_2,\cdots X_{d-1};X_d,X_{d+1})$ is symmetric about $X_1,\cdots,X_{d-1}$ and $X_d,X_{d+1}$ respectively, we get $m_2=0$.

If deg $S_{j_0}^{j_0}>1$ and $m_1\neq m_2$, then
\begin{multline*}
E(T,X_1,\cdots,X_d)-f\sum_{n-d+1}(X_1,\cdots,X_{d})=\\\overset{d}{\underset{i=1}{\prod}}(T+m_1(X_1+\cdots+\overset{\wedge}{X_i}+\cdots+X_{d}+m_2X_i)(T+n(X_1+\cdots+X_d))
\end{multline*}
since $E^{j_0}(T;X_1,X_2,\cdots X_d)-f\sum_{n-d+1}(X_1,\cdots,X_{d})$ is symmetric about $X_1,\cdots,X_d$.

The following equation is deduced from  $ \mathscr{E}(j_0) $.
\begin{multline*}
		\mathscr{\overline{E}}(j_0): \overset{d-1}{\underset{i=1}{\prod}}(T+m_1(X_1+\cdots+\overset{\wedge}{X_i}+\cdots+X_{d})+m_2X_i)(T+n(X_1+\cdots+X_d))=\\\overset{\sim}{S}_{j_0}^{j_0}(T,X_1,\cdots,X_{d+1})\overset{k}{\underset{i\neq j_0}{\prod}}S_i^{j_0}(T+(u_{i}-u_{j_0}) X_d;X_1,X_2,\cdots X_{d-1};X_d,X_{d+1})+a\sum_{n-d}(X_1,\cdots,X_{d+1}).
\end{multline*}
Taking $T=-n(X_1+\cdots+X_d)$ in $ \mathscr{\overline{E}}(j_0)$, we have
\begin{multline*} 0=\overset{\sim}{S}_{j_0}^{j_0}(-n(X_1+\cdots+X_d),X_1,\cdots,X_{d+1})\overset{k}{\underset{i\neq j_0}{\prod}}S_i^{j_0}(-n(X_1+\cdots+X_d)+\\(u_{i}-u_{j_0}) X_d;X_1,X_2,\cdots X_{d-1};X_d,X_{d+1})+a\sum_{n-d}(X_1,\cdots,X_{d+1})
\end{multline*}
which contradicts to the fact that $\sum_{n-d}(X_1,\cdots,X_{d+1})$ is irreducible when $d\geqslant 2$. So $m_1=m_2$.

If $m_1\neq0$, set
\[E^{j_0}(T;X_1,\cdots,X_d)=(T+m_1(X_1+\cdots+X_{d-1}+X_d))\overset{\sim}{E}^{j_0}(T,X_1,\cdots,X_d)\] and	
\begin{multline*}	S_{j_0}^{j_0}(T,X_1,\cdots,X_{d-1};X_d,X_{d+1})=\\(T+m_1(X_1+\cdots+X_{d-1}+X_d))\overset{\sim}{S}_{j_0}^{j_0}(T,X_1,\cdots,X_{d-1};X_d,X_{d+1}).
\end{multline*}

From  $ \mathscr{E}(j_0)$, we have
\begin{multline*} \overset{\sim}{E}^{j_0}(T,X_1,\cdots,X_d)-a\sum_{d}(X_1,\cdots,X_{d+1})=\\\overset{\sim}{S}_{j_0}^{j_0}(T,X_1,\cdots,X_{d+1})\overset{k}{\underset{i\neq j_0}{\prod}}S_i^{j_0}(T+(u_{i}-u_{j_0}) X_d;X_1,X_2,\cdots X_{d-1};X_d,X_{d+1}).
\end{multline*}
Clearly, $\overset{\sim}{E}^{j_0}(T,X_1,\cdots,X_d)$ is symmetric about $X_1,\cdots,X_d$  while ${S}_{j_0}^{j_0}(T,X_1,\cdots,X_{d-1};X_d,X_{d+1})$ is symmetric about $X_1,\cdots,X_{d-1}$ and $X_d,X_{d+1}$ respectively, so
\begin{center}	$T+m_1(X_1+\cdots+X_{d-1}+X_{d+1})|\overset{\sim}{S}_{j_0}^{j_0}(T,X_1,\cdots,X_{d-1};X_d,X_{d+1})$.
\end{center}
Then
 \begin{center}	$T+m_1(X_{1}+\cdots+X_{d-1}+X_{d+1})|E^{j_0}(T;X_1,\cdots,X_d)-a\sum_{d}(X_1,\cdots,X_{d+1})$.
\end{center}
Since the right side of above equation is symmetric about $X_1,\cdots,X_d$,
\begin{center}
$E^{j_0}(T;X_1,\cdots,X_d)-a\sum_{d}(X_1,\cdots,X_{d+1})=\overset{d}{\underset{i=1}{\prod}}(T+m_1(X_1+\cdots+\overset{\wedge}{X_i}+\cdots+X_{d+1}).$
\end{center}
Comparing the coefficient of $T^{d-1}$, we know that $(m_1(d-1))(X_1+\cdots X_d)+dm_1X_{d+1}$ is independent of $X_{d+1}$. So $m_1=0$, which contradicts to the assumption. Therefore, $m_1=m_2=0$.
\end{proof}

\begin{proposition}\label{B}
	When $a\neq0$ and deg $S_{j_0}=1$, $E^{j_0}(T;X_1,\cdots,X_d)$ is one of the following cases.
	
	\emph{(i).} $k=2$ (without loss of generality, we can assume $j_0=1$.):
	
	\[S_1(T,X_1,\cdots,X_{d+1})=T+u_2(X_1+\cdots+X_{d-1}),\]
	\begin{multline*}
	S_2(T,X_1,\cdots,X_{d+1})=\overset{d-1}{\underset{i=1}{\prod}}(T+u_2(X_1+\cdots+\overset{\wedge}{X_i}+\cdots\\+X_{d-1})+u_1X_i)(T+u_2(X_1+\cdots+X_{d-1}))+a\sum_{n-d}(X_1,\cdots,X_{d+1}),
\end{multline*}
 	\begin{multline*}
 	E(T,X_1,\cdots,X_d)=\overset{d}{\underset{i=1}{\prod}}(T+u_2(X_1+\cdots+\overset{\wedge}{X_i}+\cdots\\+X_{d})+u_1X_i)(T+u_2(X_1+\cdots+X_d))+f\sum_{n-d+1}(X_1,\cdots,X_{d})
\end{multline*}

and
\begin{multline*}
\mathscr{E}(1): E^{1}(T,X_1,\cdots,X_d)=\overset{d}{\underset{i=1}{\prod}}(T+(u_2-u_1)(X_1+\cdots\\+\overset{\wedge}{X_i}+\cdots+X_{d}))(T+(u_2-u_1)(X_1+\cdots+X_d))+f\sum_{n-d+1}(X_1,\cdots,X_{d}).
\end{multline*}

\emph{(ii).} $k=2$ (without loss of generality, we can assume $j_0=1$)	:
\[S_1(T,X_1,\cdots,X_{d+1})=T+u_1(X_1+\cdots+X_{d-1}),\]
	\[S_2(T,X_1,\cdots,X_{d+1})=(T+u_2(X_1+\cdots+X_{d-1}))^d+a\sum_{n-d}(X_1,\cdots,X_{d+1}),\]
	\begin{multline*}
		E(T,X_1,\cdots,X_d)=(T+u_1(X_1+\cdots+X_{d})((T+u_2(X_1+\cdots\\+X_{d}))^d))+f\sum_{n-d+1}(X_1,\cdots,X_{d})
\end{multline*}
and
\begin{multline*}
	\mathscr{E}(1): E^{1}(T,X_1,\cdots,X_d)=T(T+(u_2-u_1)(X_1+\cdots+X_{d}))^d+f\sum_{n-d+1}(X_1,\cdots,X_{d}).
\end{multline*}

\emph{(iii).} $k=3$:

\begin{multline*}
	E(T,X_1,\cdots,X_d)=(T+u_{j_0}(X_1+\cdots+X_d))[(T+u_q(X_1+\cdots\\+X_{d})-\beta X_{d+1})\sum_{d-1}(T+u_q(X_1+\cdots+X_d),\beta X_1,\cdots,\\\beta X_{d+1})+(\beta)^{n-d}\sum_{n-d}(X_1,\cdots,X_{d+1})]+f\sum_{n-d+1}(X_1,\cdots,X_d),
\end{multline*}	
	\[S_{j_0}(T,X_1,\cdots,X_{d+1})=T+u_{j_0}(X_1+\cdots+X_{d-1}),\]
	\[S_p(T,X_1,\cdots,X_{d+1})=T+u_q(X_1+\cdots+X_{d-1})+(u_q-u_p)(X_d+X_{d+1})\]	
and	\[S_q(T,X_1,\cdots,X_{d+1})=\sum_{d-1}(T+u_q(X_1+\cdots+X_{d-1}),(u_p-u_q)X_1,\cdots,(u_p-u_q)X_{d+1}),\]
where $p$ and $q$ are two indices different from $j_0$, $r_p=1$, $r_q=d-1$ and $\beta=u_p-u_q$.

The $\mathscr{E}(j_0)$ is :
\begin{multline*}
	E^{j_0}(T,X_1,\cdots,X_d)=T[(T+(u_q-u_{j_0})(X_1+\cdots\\+X_{d})-\beta X_{d+1})\sum_{d-1}(T+(u_q-u_{j_0})(X_1+\cdots+X_d),\beta X_1,\cdots,\\\beta X_{d+1})+(\beta)^{n-d}\sum_{n-d}(X_1,\cdots,X_{d+1})]+f\sum_{n-d+1}(X_1,\cdots,X_d).
\end{multline*}
\end{proposition}

\begin{proof}

When deg $S_{j_0}=1$, by Proposition \ref{polyan}, $S_{j_0}^{j_0}=T+m(X_1+\cdots+X_{d-1})$, where $m=\frac{b^{j_0}}{a}$.

Assume that
\begin{multline*}	E^{j_0}(T,X_1,\cdots,X_d)-f\sum_{n-d+1}(X_1,\cdots,X_{d})=(T+m(X_1+\cdots+X_{d-1}))E^{'j_0}(T,X_1,\cdots,X_d).
\end{multline*}

Then the equation $\mathscr{E}(j_0)$ is of the form
\begin{multline*}
E^{j_0}(T;X_1,X_2,\cdots X_d)=(T+m(X_{1}+\cdots+X_{d-1}))\overset{k}{\underset{i\neq j_0}{\prod}}S_i^{j_0}(T+\\(u_{i}-u_{j_0}) X_d;X_1,X_2,\cdots, X_{d-1};X_d,X_{d+1})+a[T+m(X_{1}+\cdots+\\X_{d-1})]\sum_{n-d}(X_1,\cdots,X_{d+1})+f\sum_{n-d+1}(X_1,\cdots,X_{d}).
\end{multline*}
So
\begin{multline*}
	E^{'j_0}(T,X_1,\cdots,X_d)=\overset{k}{\underset{i\neq j_0}{\prod}}S_i^{j_0}(T+(u_{i}-u_{j_0}) X_d;X_1,X_2,\cdots,\\ X_{d-1};X_d,X_{d+1})+a\sum_{n-d}(X_1,\cdots,X_{d+1}).
\end{multline*}
Taking $X_1=\cdots=X_{d-1}=0$, in \cite{Ell}[\Rmnum{3}-2], the author proved that only the following cases may happen.
\begin{enumerate}
\item[A).] $a=0$;
\item[B).] $k-1=1$;
\item[C).] $k-1=2$.
\end{enumerate}

Case A). $a=0$: By Proposition \ref{form of Si}, we get $b^{j_0}=c^{j_0}=0$. So, we reduce the question to Section 3.

Case B). $k=2$: Without loss of generality, we can assuming $j_0=1$.

If $m\neq 0$, then $E^{1}(T,X_1,\cdots,X_d)-f\sum_{n-d+1}(X_1,\cdots,X_{d})$ is symmetric about $X_1,\cdots,X_d$ and divisible by $T+m(X_1+\cdots+X_{d-1})$. So
\begin{multline*} E^{1}(T,X_1,\cdots,X_d)-f\sum_{n-d+1}(X_1,\cdots,X_{d})=\\\overset{d}{\underset{i=1}{\prod}}[T+m(X_1+\cdots+\overset{\wedge}{X_i}+\cdots+X_{d})]F(T,X_1,\cdots,X_{d}).
\end{multline*}

Since $E^{1}(T,X_1,\cdots,X_d)$ is of degree $d+1$, $F(T,X_1,\cdots,X_{d})$ is of degree $1$ and symmetric about $X_1,\cdots,X_d$. So $$F(T,X_1,\cdots,X_{d+1})=T+g(X_1+\cdots+X_d).$$

The following equation is deduced from $\mathscr{E}(1)$.

\begin{multline*}
	 \overset{d-1}{\underset{i=1}{\prod}}(T+m(X_1+\cdots+\overset{\wedge}{X_i}+\cdots+X_{d})(T+g(X_1+\cdots+X_d))=\\S_2^1(T+(u_2-u_1)X_d,X_1,\cdots,X_{d+1})+a\sum_{n-d}(X_1,\cdots,X_{d+1}).
\end{multline*}
Taking $X_1=\cdots=X_{d-1}=0$ and $T=T'+(u_1-u_2)X_d$ in above equation, we have
\begin{multline*}
	 (T'+(m+u_1-u_2)X_d)^{d-1}(T'+(g+u_1-u_2)X_d)=\\S_2(T',0,\cdots,0,X_d,X_{d+1})+a\sum_{d}(X_d,X_{d+1}).
\end{multline*}
Since the right side of the above equation is symmetric about $X_d,X_{d+1}$, we can get \[m+u_1-u_1=g+u_1-u_2=0, ~\emph{i.e.}~ m=g=u_2-u_1.\]

Thus, the equation $\mathscr{E}(1)$ is
\begin{multline*}
	 E^{1}(T,X_1,\cdots,X_d)=\overset{d}{\underset{i=1}{\prod}}(T+(u_2-u_1)(X_1+\cdots\\+\overset{\wedge}{X_i}+\cdots+X_{d})(T+(u_2-u_1)(X_1+\cdots+X_d))+f\sum_{n-d+1}(X_1,\cdots,X_{d}).
\end{multline*}

If $m=0$, then $\mathscr{E}(1)$ becomes

\begin{multline*}
E^{1}(T;X_1,X_2,\cdots X_d)=TS_2^{1}(T+(u_2-u_1)X_d;X_1,X_2,\cdots,\\ X_{d-1};X_d,X_{d+1})+aT\sum_{n-d}(X_1,\cdots,X_{d+1})+f\sum_{n-d+1}(X_1,\cdots,X_{d}).	
\end{multline*}

Suppose that
\begin{center}
	$E^{1}(T;X_1,X_2,\cdots X_d)-f\sum_{n-d+1}(X_1,\cdots,X_{d})=TE^{'1}(T,X_1,\cdots,X_d)$.
\end{center}
Then
\begin{multline*}
	E^{'1}(T,X_1,\cdots,X_d)=S_2^{1}(T+(u_2-u_1)X_d;X_1,X_2,\cdots,\\ X_{d-1};X_d,X_{d+1})+a\sum_{n-d}(X_1,\cdots,X_{d+1}).
\end{multline*}
Clearly, $E^{'1}(T,X_1,\cdots,X_d)$ is symmetric about $X_1,\cdots,X_d$. So, by Proposition 4.2 in \cite{Guy}, we have
\begin{center}
	$ E^{1'}(T;X_1,X_2,\cdots X_{d})=(T+(u_2-u_1)(X_1+\cdots+X_d))^d $.
\end{center}
The equation $\mathscr{E}(1)$ is
\begin{center} $E^{1}(T,X_1,\cdots,X_d)=T(T+(u_2-u_1)(X_1+\cdots+X_{d}))^d+f\sum_{n-d+1}(X_1,\cdots,X_{d})$.
\end{center}

Case C): $k=3$.

If $m\neq 0$, by argument above, we have
\begin{multline*} \overset{d-1}{\underset{i=1}{\prod}}(T+m(X_1+\cdots+\overset{\wedge}{X_i}+\cdots+X_{d})(T+g(X_1+\cdots+X_d))=\\\overset{3}{\underset{i\neq j_0}{\prod}}S_i^{j_0}(T+(u_{i}-u_{j_0}) X_d;X_1,X_2,\cdots, X_{d-1};X_d,X_{d+1})+a\sum_{n-d}(X_1,\cdots,X_{d+1}) .
\end{multline*}
Taking $T=-g(X_1+\cdots+X_d)$ in above equation, we get
\begin{center}
	$0=\overset{3}{\underset{i\neq j_0}{\prod}}S_i^{'j_0}(X_1,X_2,\cdots,X_{d+1})+a\sum_{n-d}(X_1,\cdots,X_{d+1})$.
\end{center}
Since $\sum_{n-d}(X_1,\cdots,X_{d+1})$ is irreducible when $d+1\geq 3$, we have $a=0$. We reduce the case to Section 3.

If $m=0$, $E^{'j_0}$ is symmetric about $X_1,\cdots,X_d$,
by Lemma 4.2.3 in \cite{Guy}, we have
\begin{multline*}
	E^{'j_0}(T,X_1,\cdots,X_d)=(T^{*}-\beta X_{d+1})\sum_{d-1}(T^{*},\beta X_1,\cdots,\\\beta X_{d+1})+(\beta)^{n-d}\sum_{n-d}(X_1,\cdots,X_{d+1}),
\end{multline*}
where $p$ and $q$ are two indices different from $j_0$, $r_p=1$, $r_q=d-1$ and $\beta=u_p-u_q$.

Now, \begin{center}
	$ E^{j_0}(T^*,X_1,\cdots,X_d)=T\sum_{d}(T^*,\beta X_1,\cdots,\beta X_d)+f\sum_{n-d+1}(X_1,\cdots,X_d)$,
\end{center}
	where $T^*=T+(u_{q}-u_{j_0})(X_1+\cdots+X_d)$,  $\beta=u_p-u_q$ and $p$ and $q$ are two indices different from $j_0$, $r_p=1$ and $r_q=d-1$.
Thus we get $\mathscr{E}(1)$ when $k=3$.
\end{proof}

\begin{proposition}\label{C}
	 When $a\neq0$ and deg $S_{j_0}>1$, then $ E^{j_0}(T,X_1,\cdots,X_d) $ is one of the following cases.
	
	\emph{(i).} $k=2$ (without loss of generality, we can assume $j_0=1$):
	
\begin{multline*}
		E(T,X_1,\cdots,X_d)=(T+u_{1}(X_1+\cdots+X_d))[(T+\\u_1(X_1+\cdots+X_{d})+\beta X_{d+1})\sum_{d-1}(T+u_q(X_1+\cdots+X_d),-\beta X_1,\cdots,-\beta X_{d+1})+\\(-\beta)^{n-d}\sum_{n-d}(X_1,\cdots,X_{d+1})]+f\sum_{n-d+1}(X_1,\cdots,X_d),	
	\end{multline*}
\begin{multline*}
		S_1(T,X_1,\cdots,X_{d+1})=[T+u_1(X_1+\cdots+X_{d-1})][T+u_2(X_1+\cdots\\+X_{d-1})+\beta (X_d+X_{d+1})],
\end{multline*}
\[S_2(T,X_1,\cdots,X_{d+1})=\sum_{d-1}(T+u_2(X_1+\cdots+X_{d-1}),-\beta X_1,\cdots,-\beta X_{d+1})\]
and
\begin{multline*}
		\mathscr{E}(1): E^{1}(T,X_1,\cdots,X_d)=T[(T+\beta (X_1+\cdots\\+X_{d+1}))\sum_{d-1}(T+\beta(X_1+\cdots+X_d),-\beta X_1,\cdots,-\beta X_{d+1})+\\(-\beta)^{n-d}\sum_{n-d}(X_1,\cdots,X_{d+1})]+f\sum_{n-d+1}(X_1,\cdots,X_d),
\end{multline*}
where $\beta=u_2-u_1$.
		
\emph{(ii).} $k=2$ (without loss of generality, we can assume $j_0=1$):
	\begin{multline*}
		E(T,X_1,\cdots,X_d)=(T+u_{1}(X_1+\cdots+X_d))[(T+u_1(X_1+\cdots\\+X_{d})+\beta X_{d+1})\sum_{d-1}(T+u_1(X_1+\cdots+X_d),-\beta X_1,\cdots,-\beta X_{d+1})+\\(-\beta)^{n-d}\sum_{n-d}(X_1,\cdots,X_{d+1})]+f\sum_{n-d+1}(X_1,\cdots,X_d),
	\end{multline*}
\begin{multline*}
		S_1(T,X_1,\cdots,X_{d+1})=\sum_{d-1}(T+u_1(X_1+\cdots\\+X_{d-1}),-\beta X_1,\cdots,-\beta X_{d+1})[T+u_1(X_1+\cdots+X_{d-1})],
	\end{multline*}
		\[S_2(T,X_1,\cdots,X_{d+1})=T+u_1(X_1+\cdots+X_{d-1})+\beta (X_d+X_{d+1})\]
and
\begin{multline*}
	\mathscr{E}(1): E^{1}(T,X_1,\cdots,X_d)=T[(T+\beta X_{d+1})\sum_{d-1}(T,-\beta X_1,\cdots,\\-\beta X_{d+1})+(-\beta)^{n-d}\sum_{n-d}(X_1,\cdots,X_{d+1})]+f\sum_{n-d+1}(X_1,\cdots,X_d),
\end{multline*}
where $\beta=u_1-u_2$.
\end{proposition}

\begin{proof}

Since deg$S_{j_0}>1$, by Proposition \ref{polyan},
\[E^{j_0}(T;X_1,\cdots,X_d)=T\overset{\sim}{E}^{j_0}(T,X_1,\cdots,X_d)\]

 and	
 	\[S_{j_0}^{j_0}(T,X_1\cdots,X_{d-1};X_d,X_{d+1})=T\overset{\sim}{S}_{j_0}^{j_0}(T,X_1,\cdots,X_{d-1};X_d,X_{d+1}).\]
 	
 Clearly, $\overset{\sim}{E}^{j_0}(T,X_1,\cdots,X_d)$ is symmetric about $X_1,\cdots,X_d$, and $\overset{\sim}{S}_{j_0}^{j_0}(T,X_1,\cdots,X_{d-1};X_d,X_{d+1})$ is symmetric about $X_1,\cdots,X_{d-1}$ and $X_d,X_{d+1}$ respectively, so we have the following equation.
 \begin{multline*}
  \overset{\sim}{E}^{j_0}(T,X_1,\cdots,X_d)=\overset{\sim}{S}_{j_0}^{j_0}(T,X_1,\cdots,X_{d-1};X_d,X_{d+1})\overset{k}{\underset{i\neq j_0}{\prod}}S_i^{j_0}(T+\\(u_{i}-u_{j_0}) X_d;X_1,X_2,\cdots X_{d-1};X_d,X_{d+1})+a\sum_{n-d}(X_1,\cdots,X_{d+1}) .
 \end{multline*}
In \cite{Guy}, the author proved the following results in Lemma 4.2.3.

 $k=2$, and deg $\overset{\sim}{S}_{j_0}^{j_0}=1$ or $d-1$. We set $j_0=1$.

 When deg $\overset{\sim}{S}_{j_0}^{j_0}=1$,
\[\overset{\sim}{S}_1^1(T,X_1,\cdots,X_{d+1})=T+\beta(X_1+\cdots+X_{d+1}) \]
 and	
\[S_2^{1}(T^{*}+(u_2-u_1)X_d,X_1,\cdots,X_{d+1})=\sum_{n-d-1}(T^{*},-\beta X_1,\cdots,-\beta X_{d+1}),\]
 where 	$T^*=T+\beta(X_1+\cdots+X_d) $ and $\beta=u_2-u_1$.

 When deg $\overset{\sim}{S}_{j_0}^{j_0}=d-1$,
\[\overset{\sim}{S}_1^1(T,X_1,\cdots,X_{d+1})=\sum_{d-1}(T,(u_2-u_1)X_1,\cdots,(u_2-u_1) X_{d+1})\]
 and
\[S_2^1(T+(u_2-u_1)X_d,X_1,\cdots,X_{d+1})=T+(u_1-u_2)X_{d+1}.\]

Then,  we deduce our conclusion from above results.
 \end{proof}

\subsection{Uniform $(d+1)$-bundle over $G(d,n)$ when $d=n-d$}


\begin{proposition}
	Taking a dualization or tensoring with line bundle if necessary, the uniform vector bundles corresponding to Proposition \ref{B} are as follows:
	\begin{center}
	$H_d^{*}(u_2)\oplus \mathcal{O}_{G(d,n)}(u_2), \mathcal{O}_{G(d,n)}(u_1)\oplus\overset{d}{\underset{i=1}{\oplus}}\mathcal{O}_{G(d,n)}(u_2), \mathcal{O}_{G(d,n)}(u_1)\oplus Q_{n-d}(u_2-1)$ or $S^{2}Q_{2}(u_2)$.
	\end{center}
\end{proposition}
\begin{proof}
Consider the standard diagram
\begin{center}
	$\xymatrix
	{
		F(d-1,d,n)\ar[dr]^{q} & F(d-1,d,d+1,n)\ar[r]^{pr_2} \ar[d]^{p}\ar[l]_{pr_1}& F(d,d+1,n)\ar[dl]^{r}\\
		& G(d,n). & \\
	}$
\end{center}
For case (i) in Proposition \ref{B}, after replacing $T$ with $T+u_2(X_1+\cdots X_d)$, we have
\begin{center} $c_{p^{*}E(-u_2)}(T)=(T+X_d)(\overset{d-1}{T\underset{i=1}{\prod}}(T+X_i)+a\sum_{n-d}(X_1,\cdots,X_{d+1}))$.
\end{center}
The HN-filtration gives the exact sequence
\begin{center}
	$ 0\longrightarrow HN^1\longrightarrow p^{*}E(-u_2)\longrightarrow F\longrightarrow 0. $
\end{center}
For line bundles $HN^1$ and $\mathscr{H}_{H_d^*}$, since $c_{HN^1}(T)=T+X_d=c_{\mathscr{H}_{H_d^*}}(T)$, we get that $HN^1\cong \mathscr{H}_{H_d^*}$. Thus $p^*E(-u_2)$ has $\mathscr{H}_{H_d^*}$ as a subbundle. By Lemma \ref{1-1lemma}, $E(-u_2)$ has $H_d^*$ as a subbundle.  Because of the splitting type, we have exact sequence
\begin{center}
	$0\longrightarrow H_d^*\longrightarrow E(-u_2)\longrightarrow \mathcal{O}_{G(d,n)}\longrightarrow 0$.
\end{center}
Therefore, $E(-u_2)\cong H_d^*\oplus \mathcal{O}_{G(d,n)}$.

For case (ii) in Proposition \ref{B}, after replacing $T$ with $T+u_1(X_1+\cdots X_d)$, we have

\begin{center}		$c_{p^{*}E(-u_1)}(T)=T((T+(u_2-u_1)(X_1+\cdots+X_{d-1}))^d+a\sum_{n-d}(X_1,\cdots,X_{d+1}))$.
	\end{center}
The HN-filtration gives the exact sequence
\begin{center}
	$ 0\longrightarrow HN^1\longrightarrow p^{*}E(-u_1)\longrightarrow F\longrightarrow 0.$
\end{center}
For line bundles $HN^1$ and $p^{*}{\mathcal{O}_{G(d,n)}}$, since $c_{HN^1}(T)=T=c_{p^{*}{\mathcal{O}_{G(d,n)}}}(T)$, we have $HN^1\cong p^{*}{\mathcal{O}_{G(d,n)}}$. Thus $p^*E$ has $p^{*}\mathcal{O}_{G(d,n)}$ as a subbundle. Applying $p_*$ to the above sequence, we have
\begin{center}
	$0\longrightarrow \mathcal{O}_{G(d,n)}\longrightarrow E(-u_1)\longrightarrow p_* F\longrightarrow 0$.
\end{center}
Comparing the splitting type, we know that $p_* F$ has the splitting type $(-1,\cdots, -1)$. By Proposition 3.1 in \cite{D-F-G}, $p_* F\cong \overset{d}{\underset{i=1}{\oplus}}\mathcal{O}_{G(d,n)}(-1)$. Therefore $E\cong \mathcal{O}_{G(d,n)}(u_1)\oplus\overset{d}{\underset{i=1}{\oplus}}\mathcal{O}_G(u_2)$.

For case (iii) in Proposition \ref{B}, taking a dualization if necessary, we may assume that $j_0=1$ or $j_0=2$.

When $j_0=1$, after replacing $T$ with $T+u_1(X_1+\cdots X_d)$, the HN-filtration gives the exact sequence
\begin{center}
	$0\longrightarrow HN^1\longrightarrow p^{*}E(-u_1)\longrightarrow F\longrightarrow 0$.
\end{center}

For line bundles $HN^1$ and $p^{*}{\mathcal{O}_{G(d,n)}}$, since $c_{HN^1}(T)=T=c_{p^{*}{\mathcal{O}_{G(d,n)}}}(T)$ , we have $HN^1\cong p^{*}{\mathcal{O}_{G(d,n)}}$ and the exact sequence
 \begin{center}
 	$0\longrightarrow p^{*}\mathcal{O}_{G(d,n)}\longrightarrow p^{*}E(-u_1)\longrightarrow F\longrightarrow 0$.
 \end{center}
Applying $p_*$ to the above sequence, we get
\begin{center}
	$0\longrightarrow \mathcal{O}_{G(d,n)}\longrightarrow E(-u_1)\longrightarrow p_*F\longrightarrow 0$.
\end{center}
Because of the splitting type of $E$, we know that $p_*F$ is a uniform vector bundle. By Theorem 1 in \cite{Guy}, we have $p_*F\cong H_{d}(-1)$ or $p_*F\cong H_{d}^*(-2)$. Therefore $E\cong \mathcal{O}_{G(d,n)}(u_1)\oplus H_{d}(u_1-1)$ or $E\cong \mathcal{O}_{G(d,n)}(u_1)\oplus H_{d}^*(u_1-2)$.

When $j_0=2$,  taking a dualization if necessary, we may assume that $r_1=1$. After replacing $T$ with $T+u_3(X_1+\cdots+X_d)$, the HN-filtration gives the exact sequences
\begin{equation}\label{d=n-d1}
0\longrightarrow HN^1\longrightarrow HN^2\longrightarrow I\longrightarrow 0
\end{equation}
and	
\begin{equation}\label{d=n-d2}
0\longrightarrow HN^2\longrightarrow p^{*}E\longrightarrow F\longrightarrow 0.
\end{equation}

We know  that
\[c_{HN^1}(T)=T-2X_{d+1}=c_{ \mathscr{H}_{Q_{n-d}}^{\otimes 2}}(T)\] and \[c_{HN^{2}/HN^{1}}(T)=T+X_{1}+X_{2}+\cdots+X_{d}=c_{p^*\mathcal{O}_{G(d,n)}(1)}(T).\]  Since $HN^1$, $ \mathscr{H}_{Q_{n-d}}^{\otimes 2}$, $HN^{2}/HN^{1}$ and $p^*\mathcal{O}_{G(d,n)}(1)$ are line bundles, we get $HN^1\cong  \mathscr{H}_{Q_{n-d}}^{\otimes 2}$ and $HN^{2}/HN^{1}\cong p^*\mathcal{O}_{G(d,n)}(1)$. Then the exact sequences (\ref{d=n-d1}) becomes
\[0\longrightarrow  \mathscr{H}_{Q_{n-d}}^{\otimes 2} \longrightarrow HN^2 \longrightarrow p^{*}\mathcal{O}_{G(d,n)}(1) \longrightarrow 0.\]

When $d=n-d>2$ and $i>0$, we have
\[ R^{i}pr_{2*} (\mathscr{H}_{Q_{n-d}}^{\otimes 2}\otimes p^{*}\mathcal{O}_{G(d,n)}(-1))=0\]
and	
\[R^{i}r_{*}(\mathcal{O}_{F(d,d+1,n)}(-2))=0\] since $F(d,d+1,n)$ can be viewed as a projective bundle over $G(d,n)$.

Combining with projection formula, we have
\begin{eqnarray*}
&&H^{1}(F(d-1,d,d+1,n),\mathscr{H}_{Q_{n-d}}^{\otimes 2}\otimes p^{*}\mathcal{O}_{G(d,n)}(-1))\\
&=&H^1(F(d,d+1,n),pr_{2*}(\mathscr{H}_{Q_{n-d}}^{\otimes 2}\otimes p^{*}\mathcal{O}_{G(d,n)}(-1)))\\
&=&H^1(F(d,d+1,n),\mathcal{O}_{F(d,d+1,n)}(-2)\otimes r^{*}\mathcal{O}_{G(d,n)}(-1))\\
&=&H^1(G(d,n),r_{*}\mathcal{O}_{F(d,d+1,n)}(-2)\otimes \mathcal{O}_{G(d,n)}(-1))\\
&=&H^1(G(d,n),0)\\
&=&0.
\end{eqnarray*}

Since \[H^{1}(F(d-1,d,d+1,n),\mathscr{H}_{Q_{n-d}}^{\otimes 2}\otimes p^{*}\mathcal{O}_{G(d,n)}(-1))=Ext^{1}(p^{*}\mathcal{O}_{G(d,n)}(1),\mathscr{H}_{Q_{n-d}}^{\otimes 2})=0,\] we have
\begin{center}
	$ HN^2\simeq \mathscr{H}_{Q_{n-d}}^{\otimes 2}\oplus p^{*}\mathcal{O}_{G(d,n)}(1) .$
\end{center}

Applying $p_* $ to (\ref{d=n-d2}) we have the exact sequence

\begin{center}
	$ 0\longrightarrow \mathcal{O}_{G(d,n)}(1)  \longrightarrow E \longrightarrow p_{*}F\longrightarrow 0 $.
\end{center}

Moreover, over $ F(d-1,d,d+1,n)$, we have the commutative diagram

\begin{center}
	$\xymatrix{
		0\ar[r] & p^{*}\mathcal{O}_{G(d,n)}(1) \ar[r] \ar[d] & p^{*}E \ar[d]^{id} \ar[r] & p^{*}p_{*}F\ar[r]\ar[d]&0\\
		0\ar[r] & HN^2 \ar[r]&p^{*}E \ar[r]&F\ar[r]&0.\\
	}$	
\end{center}

The snake lemma gives the following exact sequence:

\begin{center}
	$ 0\longrightarrow  \mathscr{H}_{Q_{n-d}}^{\otimes 2} \longrightarrow p^{*}p_{*}F \longrightarrow F \longrightarrow 0 $.
\end{center}

Restricting to the fiber of $s$ at some point $l$ in $F(d-1,d+1,n)$, we have
\[Hom(T_{F(d-1,d,d+1,n)/G(d,n)},\mathscr{H}om(\mathscr{H}_{Q_{n-d}}^{\otimes 2},F))|_{s^{-1}(l)}=0\]  \[\Rightarrow Hom(T_{F(d-1,d,d+1,n)/G(d,n)},\mathscr{H}om(\mathscr{H}_{H_{d}^{*}}^{\otimes 2},F))=0.\]

Thus, by Descente-Lemma (Lemma 2.1.2 in \cite{O-S-S}), there is a line subbundle $L$ of $p_{*}F$ such that $p^{*}L $ is isomorphic to $\mathscr{H}_{Q_{n-d}}^{\otimes 2}$, which is impossible because $p_*p^{*}L\cong L$ while $p_*\mathscr{H}_{Q_{n-d}}^{\otimes 2}=0$. Thus, $n-d=2$.

If $n-d=2$, then $d=2$ and $n=4$. The HN-filtration gives two exact sequences

\begin{equation}\label{d=n-d3}
	0\longrightarrow  \mathscr{H}_{Q_{2}}^{\otimes 2} \longrightarrow HN^2 \longrightarrow p^{*}\mathcal{O}_{G(2,4)}(1) \longrightarrow 0
\end{equation}
and
\begin{equation}\label{d=n-d4}	
0\longrightarrow  HN^2 \longrightarrow p^*E \longrightarrow \mathscr{H}_{Q_{2}}^{*\otimes 2}\otimes p^{*}\mathcal{O}_{G(2,4)}(2) \longrightarrow 0.
\end{equation}

Next, we want to consider the extension of the exact sequence (\ref{d=n-d3}) by calculating \[H^{1}(F(1,2,3,4), \mathscr{H}_{Q_{2}}^{\otimes 2}\otimes p^{*}\mathcal{O}_{G(2,4)}(-1)).\]

By the Leray spectral sequence, we have the exact sequence
\begin{multline*}
H^{1}(G(2,4),p_{*}(\mathscr{H}_{Q_{2}}^{\otimes 2}\otimes p^{*}\mathcal{O}_{G(2,4)}(-1)))\rightarrow H^{1}(F(1,2,3,4), \mathscr{H}_{Q_{2}}^{\otimes 2}\otimes p^{*}\mathcal{O}_{G(2,4)}(-1))\rightarrow \\ H^{0}(G(2,4),R^{1}p_{*}(\mathscr{H}_{Q_{2}}^{\otimes 2}\otimes p^{*}\mathcal{O}_{G(2,4)}(-1)))\rightarrow H^{2}(G(2,4),p_{*}(\mathscr{H}_{Q_{2}}^{\otimes 2}\otimes p^{*}\mathcal{O}_{G(2,4)}(-1))).
\end{multline*}

For $p_{*}\mathscr{H}_{Q_{2}}^{\otimes2}=0$, we have
\begin{eqnarray*}
&&H^{1}(F(1,2,3,4), \mathscr{H}_{Q_{2}}^{\otimes 2}\otimes p^{*}\mathcal{O}_{G(2,4)}(-1))\\
&=&H^{0}(G(2,4),R^{1}p_{*}(\mathscr{H}_{Q_{2}}^{\otimes 2}\otimes p^{*}\mathcal{O}_{G(2,4)}(-1)))\\
&=&H^{0}(G(2,4),\mathcal{O}_{G(2,4)})\\
&=&k.
\end{eqnarray*}

So $HN^{2}$ is the direct sum of line bundles or the nontrivial extension. From the argument above, $HN^{2}$ is not the trivial extension.

Applying $p_{*}$ to (\ref{d=n-d4}), we have

\begin{equation}\label{S2Q22}
0\longrightarrow p_{*}HN^{2}\longrightarrow E\longrightarrow S^{2}Q_{2}(2)\longrightarrow R^{1}p_{*}HN^{2}\longrightarrow 0.
\end{equation}

Applying $pr_{2*}$ to (\ref{d=n-d3}) and setting $M=pr_{2*}HN^{2}$, we have
\begin{equation}\label{d=n-dM}
	0\longrightarrow \mathcal{O}_{F(2,3,4)}(-2)\longrightarrow M \longrightarrow r^{*}\mathcal{O}_{G(2,4)}(1)\longrightarrow 0.
\end{equation}
By the similar argument in Lemma \ref{S^2 case}, $M$ is a nontrivial extension.

We are going to prove $r_*M=p_*HN^2=0$. Now, $r^{-1}(x)$ is a line for a point $x\in G(2,4)$. The exact sequence
\begin{center}
	$0\longrightarrow \mathcal{I}_{r^{-1}(x)} \longrightarrow \mathcal{O}_{F(2,3,4)} \stackrel{h_{x}}{\longrightarrow} \mathcal{O}_{r^{-1}(x)}\longrightarrow0$
\end{center}
holds.  Tensoring with $\mathcal{O}_{F(2,3,4)}(-2)\otimes r^{*}\mathcal{O}_{G(2,4)}(-1)$, we get the exact sequence
\begin{center}
	$0\longrightarrow \mathcal{I}_{r^{-1}(x)}\otimes \mathcal{O}_{F(2,3,4)}(-2)\otimes r^{*}\mathcal{O}_{G(2,4)}(-1) \longrightarrow \mathcal{O}_{F(2,3,4)}(-2)\otimes r^{*}\mathcal{O}_{G(2,4)}(-1) \stackrel{h_{x}}{\longrightarrow} \mathcal{O}_{r^{-1}(x)}(-2)\longrightarrow0$,
\end{center}
which induces the long exact sequence
\begin{multline*}
0=H^0(r^{-1}(x),\mathcal{O}_{r^{-1}(x)}(-2))\longrightarrow H^1(F(2,3,4),\mathcal{I}_{r^{-1}(x)}\otimes \mathcal{O}_{F(2,3,4)}(-2)) \longrightarrow \\ H^1(F(2,3,4),\mathcal{O}_{F(2,3,4)}(-2)\otimes r^{*}\mathcal{O}_{G(2,4)}(-1))\stackrel{f_x}{\longrightarrow} H^1(r^{-1}(x),\mathcal{O}_{r^{-1}(x)}(-2)).
\end{multline*}

Over $ F(2,3,4) $, we have the commutative diagram

\begin{center}
	$\xymatrix{
		0\ar[r] & \mathcal{O}_{F(2,3,4)}(-2)\otimes r^{*}\mathcal{O}_{G(2,4)}(-1) \ar[r] \ar[d] &M\otimes r^{*}\mathcal{O}_{G(2,4)}(-1) \ar[d]^{id} \ar[r] & \mathcal{O}_{F(2,3,4)}\ar[r]\ar[d]&0\\
		0\ar[r] & \mathcal{O}_{r^{-1}(x)}(-2) \ar[r]&M|_{r^{-1}(x)}\ar[r]&\mathcal{O}_{r^{-1}(x)}\ar[r]&0,\\
	}$	
\end{center}
which induces the commutative diagram
\begin{center}
	$\xymatrix{
		H^0(F(2,3,4),\mathcal{O}_{F(2,3,4)})\ar[r]^(0.4){\delta_M}\ar[d]^{h'}&H^1(F(2,3,4),\mathcal{O}_{F(2,3,4)}(-2)\otimes q^{*}\mathcal{O}_{G(2,4)}(-1))\ar[d]^{f_x}\\
		H^0(r^{-1}(x),\mathcal{O}_{r^{-1}(x)})\ar[r]^(0.45){\delta_{M|{r^{-1}(x)}}}&H^1(r^{-1}(x),\mathcal{O}_{r^{-1}(x)}(-2)).
	}$
\end{center}
Then
\[f_{x}\circ {\delta_M(1)}=\delta_{M|{r^{-1}(x)}}\circ h'(1)\]
and	
\[h'(1)=1, \delta_M(1)=t\neq 0\]
since (\ref{d=n-dM}) is the nontrivial extension.

By the similar arguments in Lemma \ref{S^2 case}, we can get that
%
%
%
$\delta_{M|{r^{-1}(x)}}\neq 0$. Thus $M|_{r^{-1}(x)}\cong \mathcal{O}_{\mathbb{P}^1}(-1)\oplus \mathcal{O}_{\mathbb{P}^1}(-1)$ for all $r$-fibers.

Since $p=r\circ pr_{2*}$ and $pr_{2*}HN^{2}$ is of splitting type $(-1,-1)$ at each fiber of $r$, we have $p_{*}HN^{2}=R^{1}p_{*}HN^{2}=0$. Therefore  $E\cong S^{2}Q_{2}(2)$ from (\ref{S2Q22}).
\end{proof}

\begin{proposition}
	The uniform vector bundle $E$ corresponding Proposition \ref{C} is isomorphic to
		$Q_{n-d}(u_2)\oplus \mathcal{O}_{G(d,n)}(u_1)$ or $Q_{n-d}^{*}(u_1)\oplus \mathcal{O}_{G(d,n)}(u_1)$.	
\end{proposition}

\begin{proof}
For case (i),
by the similar argument in Lemma \ref{(1,1,000case)}, $E\cong Q_{n-d}(u_2)\oplus \mathcal{O}_{G(d,n)}(u_1)$.

For case (ii), after replacing $T$ with $T+u_1(X_1+\cdots+X_d)$, the HN-filtration gives the exact sequence
\begin{center}
	$0\longrightarrow HN^1\longrightarrow p^{*}E(-u_1)\longrightarrow F\longrightarrow0$.
\end{center}
Since $c_{F}(T)=T+X_{d+1}=c_{\mathscr{H}_{Q_{n-d}}^{*}}(T)$, we have $F\cong \mathscr{H}_{Q_{n-d}}^{*}$. Thus $p^{*}E^{*}(u_1)$ has $\mathscr{H}_{Q_{n-d}}$ as a subbundle. By Lemma \ref{1-1lemma}, $E^{*}(u_1)$ has $Q_{n-d}$ as a subbundle. So we have the exact sequence
\begin{center}
	$0\longrightarrow Q_{n-d}\longrightarrow E^{*}(u_1)\longrightarrow N\longrightarrow0$
\end{center}
for some line bundle $N$.
Comparing the splitting type, we know that $N\cong \mathcal{O}_{G(d,n)}$. So $E\cong Q_{n-d}^{*}(u_1)\oplus \mathcal{O}_{G(d,n)}(u_1)$.
\end{proof}

\vspace{.5cm}
\textbf{Data availability} No data was used for the research described in the article.

\vspace{.5cm}
\textbf{Declarations}

\textbf{Conflict of interest}: The author states that there is no conflict of interest.

\bibliographystyle{plain}
\bibliography{ref}

\begin{thebibliography}{10}

\bibitem{Bal2}
E.~Ballico.
\newblock Uniform vector bundles on quadrics.
\newblock {\em Ann. Univ. Ferrara Sez. VII (N.S.)}, 27(1):135--146, 1982.

\bibitem{Bal}
E.~Ballico.
\newblock Uniform vector bundles of rank $(n+1) $ on $ \mathbb{P}^{n} $.
\newblock {\em Tsukuba J. Math.}, 7(2):215--226, 1983.

\bibitem{D-F-G2}
R.~Du, X.~Fang, and Y.~Gao.
\newblock Vector bundles on rational homogeneous spaces.
\newblock {\em Ann. Mat. Pur. Appl.}, 200(6):2797--2827, 2021.

\bibitem{D-F-G}
R.~Du, X.~Fang, and Y.~Gao.
\newblock Vector bundles on flag varieties.
\newblock {\em Math. Nach.}, 296(2):630--649, 2023.

\bibitem{Ele}
G.~Elencwajg.
\newblock Les fibr\'es uniformes de rang 3 sur $\mathbb{P}^2(\mathbb{C})$ sont
  homog\'enes.
\newblock {\em Math.Ann}, 231:217--227, 1978.

\bibitem{E-H-S}
G.~Elencwajg, A.~Hirschowitz, and M.~Schneider.
\newblock Les fibr\'es uniformes de rang n sur $\mathbb{P}^n(\mathbb{C})$ sont
  ceux qu'on croit.
\newblock {\em Progr. Math.}, 7:37--63, 1980.

\bibitem{Ell}
P.~Ellia.
\newblock Sur les fibr{\'e}s uniformes de rang $(n+1)$ sur $\mathbb{P}^{n} $.
\newblock {\em M{\'e}m. Soc. Math. France (N.S.)}, 7:1--60, 1982.

\bibitem{Guy}
M.~Guyot.
\newblock Caract{\'e}risation par l'uniformit{\'e} des fibr{\'e}s universels
  sur la grassmanienne.
\newblock {\em Math. Ann.}, 270(1):47--62, 1985.

\bibitem{Har}
R.~Hartshorne.
\newblock {\em Algebraic geometry.}
\newblock Springer-Verlag, New York-Heidelberg, 1977.
\newblock xvi+496 pp.

\bibitem{K-S}
Y.~Kachi and E.~Sato.
\newblock Segre’s reflexivity and an inductive characterization of
  hyperquadrics.
\newblock {\em Mem. Amer. Math. Soc.}, 160(763), 2002.

\bibitem{M-O-C2}
R.~Mu{\~{n}}oz, G.~Occhetta, and L.~E. {Sol{\'a} Conde}.
\newblock Uniform vector bundles on {F}ano manifolds and applications.
\newblock {\em J. Reine Angew. Math. (Crelles Journal)}, 664:141--162, 2012.

\bibitem{M-O-C3}
R.~Mu{\~{n}}oz, G.~Occhetta, and L.~E. {Sol{\'a} Conde}.
\newblock Splitting conjectures for uniform flag bundles.
\newblock {\em European Journal of Mathematics}, 6:430–452, 2020.

\bibitem{O-S-S}
C.~Okonek, M.~Schneider, and H.~Spindler.
\newblock {\em Vector bundles on complex projective spaces}.
\newblock Birkh$\ddot{a}$user/Springer Basel AG, Basel, 2011.
\newblock viii+239 pp.

\bibitem{Pan}
X.~Pan.
\newblock Triviality and split of vector bundles on rationally connected
  varieties.
\newblock {\em Math. Res. Lett.}, 22(2):529–547, 2015.

\bibitem{Sat}
E.~Sato.
\newblock Uniform vector bundles on a projective space.
\newblock {\em J. Math. Soc. Japan}, 28(1):123--132, 1976.

\bibitem{Sch}
R.~L.~E. Schwarzenberger.
\newblock Vector bundles on the projective plane.
\newblock {\em Proc. London Math. Soc.}, 11(3):623--640, 1961.

\bibitem{Ven}
A.~{Van de Ven}.
\newblock On uniform vector bundles.
\newblock {\em Math. Ann.}, 195:245--248, 1972.

\end{thebibliography}

\end{document}